\documentclass[11pt]{article}

\usepackage{graphicx,psfrag}

\usepackage{amssymb,amsmath,amsthm,enumerate}
\usepackage{fullpage}
\usepackage{bbm}
\newcommand{\url}{}

\usepackage{tikz}

\RequirePackage[colorlinks,citecolor=blue,urlcolor=blue]{hyperref}

\newtheorem{lemma}{Lemma}
\newtheorem{remark}{Remark}
\newtheorem{proposition}{Proposition}
\newtheorem{definition}{Definition}
\newtheorem{theorem}{Theorem}

\newtheorem{corollary}{Corollary}

\newcommand{\cA}{\mathcal{A}}
\newcommand{\cE}{\mathcal{E}}

\newcommand{\dE}{\mathbb {E}}

\newcommand{\dN}{\mathbb {N}}

\newcommand{\dR}{\mathbb {R}}
\newcommand{\dC}{\mathbb {C}}

\newcommand{\cB}{\mathcal{B}}
\newcommand{\cG}{\mathcal {G}}

\newcommand{\cP}{\mathcal {P}}

\newcommand{\cM}{\mathcal {M}}

\newcommand{\INT}[1]{{{\left[ \hspace{-1pt} \left[ #1 \right] \hspace{-1pt} \right] }}} 

\newcommand{\AND}{\quad \mathrm{and} \quad}

\newcommand{\BB}{\mathfrak{b}}

\definecolor{darkred}{rgb}{0.9,0,0.3}

\begin{document}

\title{Tensorial free convolution, \\ semicircular, free Poisson and R-transform in high order}
\author{Rémi Bonnin
\thanks{Aix-Marseille University \& Ecole Normale Supérieure, France. Email: \href{remi.bonnin@ens.psl.eu}{remi.bonnin@ens.psl.eu}} }

\maketitle

\begin{abstract}
This work builds on our previous developments regarding a notion of freeness for tensors. We aim to establish a tensorial free convolution. We aggregate all connected rooted $p$-regular trace invariants with a fixed number of vertices into moments. 
First, we define higher-order analogues of the semicircular (or Wigner) law and the free Poisson (or Marčenko-Pastur) law, giving their moments and free cumulants. We prove the convergence of the moments of a Wishart-type tensor to the ones of the free Poisson law and recall the convergence of the moments of a Wigner tensor to the moments of the semicircular law. We also present a free Central Limit Theorem in this context. 
Next, we introduce a tensorial free convolution at the level of exponentially bounded moment and cumulant sequences.
\end{abstract}

\section{Introduction}

The foundational notion of freeness was introduced several decades ago by Dan-Virgil Voiculescu \cite{MR1094052, MR1217253}. This groundbreaking concept has given rise to a substantial body of work, notably through the contributions of Speicher, Nica, and Mingo, who have significantly advanced the development of the theory \cite{nica2018lectures, zbMATH06684673, mingonica2004}. Freeness has found numerous applications across diverse fields, including (quantum) group theory, quantum information, statistical inference and many others, see for instance \cite{nechitacollins2015, voiculescu2019hydrodynamicexercisefreeprobability,collins2024freecumulantsfreenessunitarily,Bandeira_2023}.

One of the most powerful tools provided by freeness is free convolution, which describes the distribution of the sum (or product) of two freely independent elements in terms of the individual distributions of each element \cite{VOICULESCU1986323, Speicher1994, BelinschiMaiSpeicher+2017+21+53}. Our primary goal in this paper is to extend the concept of free convolution to the setting of tensorial freeness. From there, we can explore its implications for specific distributions. In particular, two landmark distributions in free probability theory are the semicircular (or Wigner) law and the free Poisson (or Marčenko-Pastur) law. Freeness plays a crucial role in the study of random matrices, a key area within the ecosystem of noncommutative probability spaces \cite{zbMATH06684673, andersongz2010,aumale2021}. Semicircular and free Poisson laws arise notably as limiting distributions for large Wigner and Wishart matrices respectively, see \cite{eugenewigner, marcenkopastur1967, pastursccherbina2011, BaiSilv,bordenavechafai2011}.

In many respects, random tensors generalize random matrices to higher orders. This field is currently a vibrant area of research, driven by numerous applications in physics and computer science \cite{arous2018landscapespikedtensormodel,Jagannath_2020,Benedetti_2021,avohou2024countingunotimesrotimesonotimes,gurau2024quantumgravityrandomtensors}. In this paper, we aim to establish analogous results within the framework of tensorial freeness. Specifically, we proposed a notion of freeness for tensors in \cite{bonnin2024freenesstensors}, to which we refer for more detailed introduction. Here, we continue developing the foundations of tensorial free probability theory, with a particular focus on the concept of tensorial free convolution. The principal contributions of this work are the following :
\begin{itemize}
    \item[-] We define the moments and the free cumulants associated to a tensor.
    \item[-] We define the semicircular and the free Poisson laws of higher order, giving their moments and free cumulants.
    \item[-] We prove the convergence of the moments of a Wishart tensor towards the ones of the free Poisson law of high order and we recall the convergence of the moments of a Wigner's tensor towards the ones of the semicircular of high order and the free Central Limit Theorem.
    \item[-] We develop the tensorial free additive convolution for exponentially bounded sequences.
\end{itemize}
We do not use the term "higher-order freeness" as this designation is already attributed to a different concept \cite{collinsmingospeicher2007}. In the rest of this introduction, we first recall briefly the background about trace invariants and non-crossing poset (formal definitions are postponed to Section \ref{sec:mom}) and then we will state our main results.
\paragraph{Maps of tensors.}
This work will deal with real tensors that is an element of the vector space $\cE^N_p = \dR^N \otimes \cdots \otimes \dR^N$, where $p \geq 1$ is called the order, the $p$ different copies of $\dR^N$ are called the legs and $N$ is called the dimension of the legs of the tensor. A tensor $T$ of order $p$ is represented by a {\em map} with a single vertex with $p$ boundary edges, see Figure \ref{fig:map1}, where a map is a graph with a cyclic order of the edges around the vertices, the interested reader may see [\cite{bonnin2024freenesstensors}, Section 2.1]. The symmetric group $S_p$ acts on the tensors by permutation of the legs, that is $T^{\sigma}$ is the order $p$ tensor with entries 
$(T^{\sigma})_{i_1,\ldots,i_p} := T_{i_{\sigma(1)},\ldots,i_{\sigma(p)}},$
and a real tensor is said {\em symmetric} if $T=T^{\sigma}$ for all $\sigma \in S_p$. 
A contraction of two tensors $T_1$ (of order $p$) and $T_2$ (of order $q$) along $r$ legs of dimension $N$, defined as
$$ (T_1 \star_r T_2)_{i_1,\ldots,i_{p+q-2r}} := \sum_{1 \leq j_1,\ldots,j_r \leq N} (T_1)_{i_1,\ldots,i_{p-r},j_1,\ldots,j_r} (T_2)_{j_1,\ldots,j_r,i_{p-r+1},\ldots,i_{p+q-2r}}, $$
is then a map with two vertices obtained by forming edges between the contracted legs of the tensors. We can contract along the legs we want by acting with a $\sigma$ on $T$. We used here the notation of \cite{bandeira2024geometricperspectiveinjectivenorm}. More generally, let $\BB$ be a map with vertex set $V$, edge set $E$ and with $q\geq 1$ boundary edges say $\partial = (e_1,\ldots,e_q)$. Then if $(T_v)_{v \in V}$ is a collection of tensors where the order of $T_v$ is the degree of $v$ in $\BB$, we can define the tensor in $\cE^N_q$, for $i \in \INT{N}^{\partial}$, 
\begin{equation}\label{eq:trace1}
\BB ( (T_v)_{ v\in V} )_{i_{\partial}} =  \sum_{i \in \INT{N}^{E}} \prod_{v \in V} (T_v)_{i_{\partial v}},
\end{equation}
where $\partial v$ is the sequence of neighboring  edges and boundary edges of $v$. The linear combination of these maps of tensors of possibly different order encode all possible ways to contract the tensors and can be thought as the extension of the matrix polynomials in the matrix case, some examples are given in Section \ref{sec:mom}. 
If the map $\BB$ has no boundary edge we call it a {\em trace map}. For $(T_v)_{v \in V}$ a collection of tensors where the order of $T_v$ is the degree of $v$ in $\BB$, we define the {\em trace invariant} associated to this trace map as
\begin{equation}\label{eq:trace2}
\BB( (T_v)_{ v\in V} )  =  \frac{1}{N^\gamma} \sum_{i \in \INT{N}^{E}} \prod_{v \in V} (T_v)_{i_{\partial v}},
\end{equation}
where $\gamma$ is the number of connected components of $\BB$. 
The {\em distribution} of a collection of tensors $\cA  = \{ A_1,\ldots, A_n \}$ of possibly different orders, $A_k \in \cE^N_{p_k}$ is the collection of all trace maps $\BB ( (T_v)_{v \in V})$ with $T_v \in \cA$ and $\BB$ with compatible degrees. In particular, if $T \in \cE^N_p$, the {\em distribution of $T$} is the collection of trace invariants for {\em $p$-regular} trace maps, that is all vertices of the map belong to exactly $p$ edges.
\par We introduce now informally three types of $p$-regular trace maps useful for the rest of the introduction. A {\em melon} is a map with two vertices and $p$ edges between them. If $p$ is even, a {\em bouquet} is a map with one vertex and $p/2$ self-loops, and a {\em multi-cycle} is a map with $n\geq 1$ vertices $p/2$ edges between the $i$-th and the $i+1$-th ones for all $1\leq i \leq n$ ($n+1=1$). For $p$ odd, a variant of an {\em odd multi-cycle} is presented in Section \ref{sec:maps}.

\paragraph{Non-crossing poset.} The notion of freeness detailed in \cite{bonnin2024freenesstensors} arises from a poset (partially ordered set) on maps, called by analogy the poset of non-crossing partitions on maps, or informally the non-crossing poset. Loosely speaking, we say that a map $\BB'$ is smaller than a map $\BB$ if they differ by a switch, that is exchanging the ending point of two edges in $\BB$, and $\BB'$ has one more connected component. By transitivity, this gives a poset on the set of trace maps with a given number of vertices and a given sequence of degrees. The formal definitions are recalled in Section \ref{sec:mom} and an example is given in Figure \ref{fig:poset}. By Moebius inversion, it is possible to define free cumulants associated to this poset characterized by the implicit relation
\begin{equation}\label{eq:MCbasic}
    \BB((T_v)_v) = \sum_{\BB'\leq \BB} \kappa_{\BB'}((T_v)_v).
\end{equation}
There is a notion of freeness associated to this poset where importantly the distribution of a family of freely independent families is characterized by the individual distribution of the families. We also proved that the vanishing of mixed cumulants characterized freeness, but only in the case of even families for technical issues, as we will discuss later. We are now ready to go into the main results of this work.

\subsection{Main results}
Let $\cB^{(p)}_n$ be the set of connected rooted $p$-regular trace maps with $n$ (unlabeled) vertices. For a tensor $T$ of order $p$, and $n\geq 1$, we define $m_n(T)$, the $n$-th moment of $T$, as the sum of $\BB(T,\ldots,T)$ over all $\BB \in \cB^{(p)}_n$, where all vertices are decorated by $T$. Similarly, $\kappa_n(T)$ the $n$-th free cumulant of $T$ is the sum of $\kappa_{\BB}(T,\ldots,T)$ over all $\BB \in \cB^{(p)}_n$. We set $m_0(T)=\kappa_0(T)=1$. We then define the corresponding formal power series in $\dC\INT{z}$ :
$$ M_{T}(z) := \sum_{n\geq 0} m_n(T) z^n \quad \AND \quad C_{T}(z) := \sum_{n\geq 0} \kappa_n(T) z^n.$$
\begin{theorem}[Analytic moment-cumulant formula]\label{thm:MC}
    We have as functional relation in $\dC\INT{z}$,
    $$M_{T}(z) = C_{T}\left( z M_T(z)^{p/2} \right) .$$
\end{theorem}
Note that this relation is always in $\dC\INT{z}$ and not in $\dC\INT{\sqrt{z}}$. Indeed if $p$ is odd, there is no $p$-regular trace map with $2n+1$ vertices, as it would need to have $p(2n+1)/2$ edges. Therefore, all the odd moments are zero, and consequently all odd free cumulants vanish by Moebius inversion. 
\par If $a$ is a distribution on the $p$-regular trace maps (that is $a : \cup_n \cB^{(p)}_n \mapsto \mathbb{R} $ satisfying multilinearity, morphism property and $a(\emptyset)=1$), we may speak of its moments $m_n(a)$ as the sum of $a(\BB)$ over all $\BB \in \cB^{(p)}_n$. By linearity we can equivalently define $a$ by giving $a(\kappa_{\BB})$.

\paragraph{High order laws.} For $p\geq 1$, we define the two following distributions $sc_p$ and $f_{p,t}$ on the $p$-regular combinatorial maps by
$$sc_p(\kappa_{\BB})=\frac{1}{(p-1)!} \mbox{   if } \BB \mbox{ is a melon and } 0 \mbox{ otherwise,}$$
$$ f_{2p,t}(\kappa_{\BB})= \frac{t}{ p!^{n} } \mbox{ if } \BB \mbox{ is a multi-cycle of length } n \mbox{ and } 0 \mbox{ otherwise,} $$
The case of $f_{p,t}$ for odd $p$ is treated in Section \ref{sec:wi} but not stated here for ease of notation.
\begin{theorem}[Semicircular law in high order]
    For $p \geq 1$, the distribution $sc_p$ has for moments
    \[ m_n(sc_p) = \frac{1}{pn/2 +1}\binom{pn/2 +1}{n/2} \mbox{ if $n$ even} \quad \AND \quad m_n(sc_p)=0 \mbox{ if $n$ odd}, \]
    and for free cumulants
    \[ \kappa_n(sc_p)=1 \mbox{ if $n=0$ or $n=2$} \quad \AND \quad \kappa_n(sc_p)=0 \mbox{ otherwise}. \]
    Moreover, the sequence $(m_n(sc_p))_{n \geq 1}$ is the sequence of moments of a compactly supported probability measure on $\dR$, denoted $\mu_p$ and called high order semicircular law.
\end{theorem}
The number $F_p(k)=\frac{1}{pk+1}\binom{pk+1}{k}$ is the $k$-th Fuss-Catalan number of order $p$. They extend the Catalan numbers for $p\geq2$.
\begin{theorem}[Free Poisson law in high order]
    For $p\geq 1$, the distribution $f_{2p,t}$ has for moments and for free cumulants: $m_1(f_{2p,t})=\kappa_1(f_{2p,t})=1$ and for all $n\geq 1$,
        \[ m_n(f_{2p,t}) = \sum_{b=1}^n \frac{1}{b} \binom{n-1}{b-1} \binom{pn}{b-1} t^b \quad \AND \quad \kappa_n(f_{2p,t})=t . \]
    
\noindent Moreover, $(m_n(f_{2p,t}))_{n\geq 1}$ is the sequence of moments of a compactly supported probability measure on $\dR$, denoted $\nu_{2p,t}$ and called high order free Poisson law.
\end{theorem}
The polynomial $F_{p,k}(t)=\sum_{b=1}^k \frac{1}{b} \binom{k-1}{b-1} \binom{pk}{b-1} t^b$ is called the $(k+1)$-th Fuss-Narayana polynomial of order $p$ and parameter $t$. In particular, we have that $F_{p,k}(1)=F_{p+1}(k)$. Interestingly, we find relations between the high order semicircular and free Poisson of parameter $t$ when $t=1$ and when $t \rightarrow \infty$, extending the ones known for $p=2$, see Remark \ref{rk:t_inf}.
\par We may also define $\omega_{p,t}$ the Marčenko-Pastur law of parameter $\tau$ as the free Poisson law of parameter $1/ \tau$ dilated by a factor $1 / \tau^{p/2}$. In the matrix case, that is associated to the densities
$$ d \nu_{2,t}(x)=(1-t)_{+} \delta_0 + \frac{\sqrt{4t-(x-1-t)^2} }{2\pi x} dx, $$
$$ d \omega_{2,\tau}(x)=(1-\frac{1}{\tau})_{+} \delta_0 + \frac{\sqrt{4\tau-(x-1-\tau)^2}}{\tau \times 2\pi x } dx,$$
which have both moments given by some Narayana numbers. This is due to a symmetry in these numbers changing $b$ in $n-b$ which does not subsist when $p>2$, see Remark \ref{rem:MP} for more details. That is why we chose this convention in the definitions of free Poisson and Marčenko-Pastur laws.

\paragraph{Convergence results.} The analogy between these laws in high order and the usual ones is not simply justified by the extension of the moments and free cumulants to higher $p$. We have also the main classical results of convergence which remains true. In particular, we have convergence of the moments of a Wigner tensor towards the ones of the higher order semicircular law, convergence of the moments of a Wishart tensor towards the ones of the higher order free Poisson law, and a free Central Limit Theorem. 
\par Formal definitions of a Wigner and Wishart tensor will come in Section \ref{sec:wi}. Just in some words, a Wigner tensor is of the form
$$ \mathbf{W}_N :=\frac{X}{N^{\frac{p-1}{2}}},$$
where $X$ is a real symmetric $p$-order tensor having entries with mean $0$, variance $\frac{1}{(p-1)!}$ and bounded moments. This is in particular the case when $X$ belongs to the Gaussian Orthogonal Tensor Ensemble. On the other hand, for $p=2q$, a $p$-order Wishart tensor is a real symmetric tensor of the form 
$$ \mathcal{W}^k_N := \frac{x_1^{\underline{\otimes} 2}+\ldots+x_k^{\underline{\otimes} 2}}{N^{p/2}},$$
where the $x_{i}$ are $q$-order tensors with {\em i.i.d} entries with mean $0$, variance $\frac{1}{(p!)^{1/p}}$ and bounded moments and $x_1^{\underline{\otimes} 2}$ is a symmetrization of $x_i \otimes x_i$. All of that is for even $p$, see the odd case and formal definitions in Section \ref{sec:wi}. We then have the following convergences.
\begin{theorem}[High order Wigner Theorem]\label{thm:wigner}
    Let $\mathbf{W}^{p}_N$ be a $p$-order Wigner tensor. 
    Then, when $N \rightarrow \infty$, 
    $$ \mathbb{E} [m_n(\mathbf{W}_N)] = m_n(sc_{p}) + \mathcal{O}(1/N) \quad \AND \quad \mathrm{Var}[m_n(\mathbf{W}_N)] =\mathcal{O}(1/N^2) .$$
\end{theorem}
\begin{theorem}[High order Marčenko-Pastur Theorem]\label{thm:Marčenko}
    Let $\mathcal{W}^{p,k}_N$ be a $p$-order Wishart tensor such that $k_N / N =t + \mathcal{O}(N^{-1})$ with $t \in (0,\infty)$. Then, when $N \rightarrow \infty$, 
    $$ \mathbb{E} [m_n(\mathcal{W}_N)] = m_n(f_{p, t}) + \mathcal{O}(1/N) \quad \AND \quad \mathrm{Var}[m_n(\mathcal{W}_N)] =\mathcal{O}(1/N^2) .$$
\end{theorem}
Theorem \ref{thm:Marčenko} will be proven in Section \ref{sec:MPconv}. Theorem \ref{thm:wigner} has already been proved in \cite{gurau2020generalizationwignersemicirclelaw} in the Gaussian case and in \cite{bonnin2024universalitywignerguraulimitrandom} for the general case. We will give a reminder of the ideas of the proof in Section \ref{sec:SCconv}. Note that a stronger result can in fact be stated: the convergence is in distribution (at the level of the individual maps), which is given in [\cite{bonnin2024universalitywignerguraulimitrandom} Lemma 3.16] for Wigner, and can be easily deduced from the proof for Wishart. This is the content of the following Lemma \ref{lem:ind}, but will not be our interest in this paper as we consider the aggregated moments and not the distribution of the tensor.
\begin{lemma}\label{lem:ind}
    We have convergence in distribution of $\mathbf{W}^{p}_N$ (respectively $\mathcal{W}^{p,k}_N$) towards $sc_p$ (respectively $f_{p,t}$), that is for all trace map $\BB$, $\BB(\mathbf{W}) \rightarrow sc_p(\BB)$ (resp. $\mathcal{W}$ and $b$).
\end{lemma}
The last convergence result we state now is the free Central Limit Theorem for tensors.
\begin{theorem}[Free CLT]\label{thm:clt}
    Let $p\geq 2$ be a fixed even integer and $(T_i)_{i\geq 1}$ a collection of freely independent tensors of order $p$ such that 
    $$ m_1(T_i)=0 \quad \AND \quad m_2(T_i)=1 \quad \AND \quad \forall i, \forall \BB \in \cup_n \cB_n, \vert \BB(T_i) \vert \leq C(\BB) . $$
    Denote $T^k= \frac{1}{\sqrt{k}} \sum_{i=1}^k T_i$. Then when $k$ goes to infinity, 
    $$ m_n(T^k) \rightarrow m_n(sc_p).$$
\end{theorem}
We initially presented this result in \cite{bonnin2024freenesstensors} (Section 3.3), but not considering the sum over connected maps having the same number of vertices. We will provide the result in this setting in Section \ref{sec:laws}.

\paragraph{Free convolution of tensors.} For the rest of this introduction, we fix $p\geq 2$ an even integer, because the characterization of freeness by vanishing of mixed cumulants is proved in that setting. Adding free cumulants always defines an operation of algebraic convolution of exponentially
bounded moment sequences, denoted $\oplus_p$. It does not, however, always produce the moments of a probability
measure, see Remark \ref{rem:negative_moments}. We can derive using Theorem \ref{thm:MC} the law of the sum of two semicircular or free Poisson freely independent. 
\begin{corollary}
    The free convolution of two freely independent semicirculars of order $p$ is a semicircular of order $p$ dilated by a factor $\sqrt{2}$. That is,
    \[ m_n(sc_p \oplus_p sc_p)= 2^{n/2} m_n(sc_p) . \]
    It is still the sequence of moments of the measure $\mu_p$ dilated by $\sqrt{2}$.
\end{corollary}
\begin{corollary}
    The free convolution of two freely independent free Poisson of order $p$ and parameters $t$ and $t'$ is a free Poisson of order $p$ and parameter $t+t'$. That is,
    $$ m_n(f_{p,t} \oplus_p f_{p,t'} )= m_n(f_{p,t+t'}). $$
\end{corollary}
For an exponentially bounded moment sequence, define 
$$ R(z) := \frac{C(z) -1}{z}. $$
This $R$-transform is additive for the sum of two free distributions.
We define also as Laurent series at infinity,
\[ G(z):=\frac{1}{z}M\left(\frac{1}{z}\right)^{p/2} , \qquad K(z) := \frac{C(z)^{p/2}}{z} .  \]
Then $G$ and $K$ are compositional inverse.
\begin{theorem}[Subordination property] \label{thm:subordination}
    Let $a,b$ be two freely independent distributions with exponentially bounded moments and $c=a \oplus_p b$. As a formal Laurent-series identity, for all $\vert z \vert$ sufficiently large,
    $$ G_{c}(z) = G_{a}\left( \omega(z)^{p/2}\right), $$
    where $\omega(z)=z^{2/p}- G_{c}(z) R_{b}(G_{c}(z))/G_{c}(z)^{2/p}$. The choice of both $z^{2/p}$ and $G_{c}(z)^{2/p}$ is determined by the unique analytic $(p/2)$-th root of $zG_c(z)$ tending to $1$ as $z \to \infty$.
\end{theorem}
The details and the proof are given in the end of Section \ref{sec:conv}. Note that this relation simplifies when $a$ or $b$ is semicircular, see Remark \ref{rmk:subsc}.

\par These results also open paths to treat some questions about high dimension statistics for random tensors. One example of such an application of free convolution, in the matrix case, is given by the work done in \cite{capitaine2016spectrumdeformedrandommatrices}.

\subsection{Organization of the paper}
In Section \ref{sec:mom}, we introduce the combinatorial formalism and prove the analytic moment-cumulant formula. In Section \ref{sec:laws}, we define Wigner and Wishart tensors and prove convergence of their aggregated moments. In Section \ref{sec:conv}, we define the algebraic cumulant convolution and give basic examples.

\subsection{Acknowledgements}
The author wishes to extend his warm gratitude to Charles Bordenave and Djalil Chafaï for the wealth of explanations and ideas shared during their inspiring discussions. He is also very grateful to the anonymous referee who participate to the improvement of this work.


\section{Moments and free cumulants for tensors}\label{sec:mom}
An important difference is that when \cite{bonnin2024freenesstensors} and \cite{kunisky2024tensorcumulantsstatisticalinference} only consider the distribution of the trace invariants individually, we will here define the moments of a tensor on a coarser collection of invariants as done in \cite{gurau2020generalizationwignersemicirclelaw} and \cite{bonnin2024universalitywignerguraulimitrandom}.

\subsection{Maps of tensors} \label{sec:maps}
For integer $n \geq 1$, we set $\INT{n} = \{1,\ldots,n\}$.
\paragraph{Combinatorial maps.} We recall some notions already introduced in \cite{bonnin2024freenesstensors} and we refer to it for more details. The formalism of {\em combinatorial maps} is a way to encode finite graphs $\BB$ equipped with an order of edges attached to each vertex, by a pair of permutations. For $m$ even integer, a combinatorial map $\BB = (\pi,\alpha)$ with $\pi,\alpha \in S_m$ has $m/2$ edges and has vertex set $V(\BB) = \INT{n}$ where $n$ is the number of cycles of $\pi$. The set $\vec E (\BB) = \INT{m}$ are the directed edges (or half-edges), $\pi$ has $n$ cycles ordered by least elements which are the directed edges attached to each vertex and $\alpha$ is an involution without fixed point whose $m/2$ cycles of length $2$ are identified as $E(\BB)$ the edges of $\BB$. 
For $v \in V$, we denote by $\partial v = (e_1,\ldots,e_p) \in \vec E(\BB)^p$ the cycle of $\pi$ associated to $v$. We always choose $e_1$ such that $e_1 = \min \partial v$. The {\em degree} of $v \in V$, $\deg(v)$ is the length of the cycle, that is $p$. We denote by $\cM_0$ the set of such combinatorial maps with no boundaries, which we call {\em trace maps}. 

\par It is also possible to introduce $\cM_q$ the set of combinatorial maps with $q$ boundaries, with the only difference that $m$ is no more necessarily even and $\alpha$ has $q$ fixed points which are the boundaries of the maps. A tensor $T$ of order $p$ is then represented as a map with one vertex decorated by $T$ and $p$ boundary edges. The set of all combinatorial maps is denoted $\cM = \sqcup_q \cM_q$. This will not be the center of our interest in this work so we do not give more details.

\par Some examples are given in Figure \ref{fig:map1}. In particular, we describe some simple maps appearing in the sequel. For $p \geq 1$, the {\em melon} maps of degree $p$ (or {\em Frobenius pair} in the terminology of \cite{kunisky2024tensorcumulantsstatisticalinference}), denoted $\mathfrak{f}_p^{\sigma}$ for some $\sigma$, are the maps with two vertices and $p$ edges between them : $\pi  =(1,\ldots,p)(p+1,\ldots,2p)$ and $\alpha=(1,p+\sigma(1))\cdots(p,p+\sigma(p))$. There are $(p-1)!$ such maps non equivalent, the canonical one being again $\mathfrak{f}_p^{\mathrm{id}}$. For $p = 2t$ even, the {\em bouquet} maps, denoted $\BB_p^{\sigma}$ for some $\sigma$, are the ones with a single vertex and $t$ self-loops : $\pi  =(1,\ldots,p)$ is a cycle and $\alpha = (\sigma(1) , \sigma(2)) \cdots (\sigma(p-1), \sigma(p))$. The canonical one is $\BB_p^{\mathrm{id}}$. It is a particular case of {\em multi-cycle} map, that is $n\geq 1$ vertices $p/2$ edges between the $i$-th and the $i+1$-th ones for all $1\leq i \leq n$ ($n+1=1$). For a given choice of inputs and outputs edges, there are $[(p/2)!]^n$ such maps. For $p$ odd, an {\em odd multi-cycle} is a map with $2n$ vertices and $(p+1)/2$ edges between the $2i-1$-th and the $2i$-th ones, $(p-1)/2$ edges between the $2i$-th and the $2i+1$-th ones, for all $1\leq i \leq n$ ($2n+1=1$).
\begin{figure}
\begin{minipage}[c]{.46\linewidth}
    \centering
    \begin{tikzpicture}[scale=0.5]
    \filldraw[gray] (0,0) circle (5pt);
    \draw (0,0) node[left] {$M$};
    \draw (0,0) .. controls (-1,1.5) and (1,1.5) .. (0,0);
    
    \draw (1.5,0) node[right] {$\frac{1}{N} \mathrm{Tr} (M)$};
\end{tikzpicture}
\end{minipage} \hfill 
    \begin{minipage}[c]{.46\linewidth}
    \centering
    \begin{tikzpicture}[scale=0.5]
    \filldraw[gray] (0,0) circle (5pt);
    \filldraw[gray] (2,0) circle (5pt);
    \draw (0,0) node[below] {$M_1$};
    \draw (2,0) node[below] {$M_2$};
    \draw (0,0) -- (2,0);
    \draw (-1,0) -- (0,0);
    \draw (2,0) -- (3,0);
    
    \draw (5,0) node[right] {$M_1 \times M_2$};
\end{tikzpicture}
\end{minipage} \hfill 
\centering
    \begin{tikzpicture}
        \draw [dashed] (0,0) -- (10,0);
    \end{tikzpicture} \hfill
\begin{minipage}[c]{.46\linewidth}
    \centering
    \begin{tikzpicture}[scale=0.6]

    \filldraw[gray] (3,0) circle (5pt);
    \draw (3,0) .. controls (4.618,1.176) and (4.618,-1.176) .. (3,0);
    \draw (3,0) .. controls (4.618,1.176) and (2.382,1.902) .. (3,0);
    \draw (3,0) .. controls (4.618,-1.176) and (2.382,-1.902) .. (3,0);
    \draw (3,0) .. controls (2.382,1.902) and (1,0) .. (3,0);
    \draw (3,0) .. controls (2.382,-1.902) and (1,0) .. (3,0);

    \draw (6,0) node {$\BB_p^{\mathrm{id}}(T)$};

\end{tikzpicture}
    \end{minipage} \hfill 
\begin{minipage}[c]{.46\linewidth}
    \centering
    \begin{tikzpicture}[scale=0.4]
    \filldraw[gray] (-2,0) circle (5pt);
    \filldraw[gray] (2,0) circle (5pt);
    \draw (-2,0) node[below] {$T$};
    \draw (2,0) node[below] {$T$};
    \draw (-2,0) -- (2,0);
    \draw (-2,0) .. controls (-0.6,1) and (0.6,1) .. (2,0);
    \draw (-2,0) .. controls (-0.6,-1) and (0.6,-1) .. (2,0);
    \draw (-2,0) .. controls (-0.6,2) and (0.6,2) .. (2,0);
    \draw (-2,0) .. controls (-0.6,-2) and (0.6,-2) .. (2,0);
    
    \draw (4,0) node[right] {$\mathfrak{f}_p^{\mathrm{id}}(T)$};
\end{tikzpicture}
    \end{minipage} \hfill 
\centering
    \begin{tikzpicture}
        \draw [dashed] (0,0) -- (10,0);
    \end{tikzpicture} \hfill
\begin{minipage}[c]{.46\linewidth}
    \centering
    \begin{tikzpicture}[scale=0.3]
    \filldraw[gray] (0,0) circle (5pt);
    \filldraw[gray] (1.5,1.5) circle (5pt);
    \filldraw[gray] (1.5,-1.5) circle (5pt);
    \filldraw[gray] (-1.5,1.5) circle (5pt);
    \filldraw[gray] (-1.5,-1.5) circle (5pt);
    \draw (0,0) node[above] {$T$};
    \draw (1.5,1.5) node[above] {$U$};
    \draw (1.5,-1.5) node[below] {$U$} ;
    \draw (-1.5,1.5) node[above] {$U$} ;
    \draw (-1.5,-1.5) node[below] {$U$} ;
    
    \draw (0,0) -- (3,3) ;
    \draw (0,0) -- (3,-3) ;
    \draw (0,0) -- (-3,3) ;
    \draw (0,0) -- (-3,-3) ;

    \draw (6,0) node {$T\cdot U^p$};
\end{tikzpicture}
    \end{minipage}
\begin{minipage}[c]{.46\linewidth}
    \centering
    \begin{tikzpicture}[scale=0.5]
    \filldraw[gray] (1,0) circle (5pt);
    \filldraw[gray] (2,1) circle (5pt);
    \filldraw[gray] (3,0) circle (5pt);
    \filldraw[gray] (2,-1) circle (5pt);
    \draw (1,0) node[below] {$T$};
    \draw (2,1) node[right] {$u$};
    \draw (3,0) node[right] {$v$};
    \draw (2,-1) node[right] {$w$};
    \draw (-1,0) -- (1,0);
    \draw (1,0) -- (2,1);
    \draw (1,0) -- (3,0);
    \draw (1,0) -- (2,-1);
    
    \draw (5,0) node[right] {$T.(u,v,w)$};
\end{tikzpicture}
\end{minipage} \hfill
    \caption{Some tensor maps.}
    \label{fig:map1} \hfill
\end{figure}
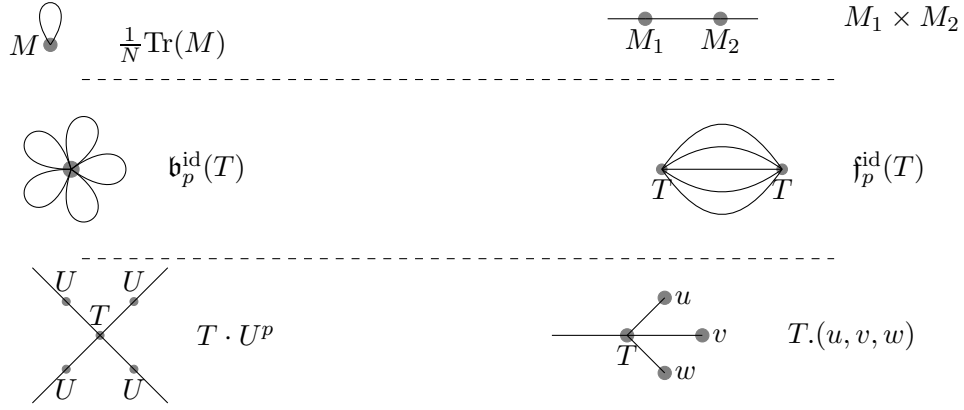

For a combinatorial map $\BB=(\pi,\alpha)$ we denote $\# \BB$ the number of vertices of $\BB$ (number of cycles of $\pi$) and $\gamma(\BB)$ the number of connected components if $\BB$ is a trace map, $\gamma(\BB)=0$ otherwise. Furthermore, for a given $p$ we say that a map is {\em $p$-regular} if $\pi$ has only cycles of length $p$ (all vertices are $p$-valent).

\paragraph{Tensor maps.} For $N \geq 1$, we set $\cE^N_p := (\mathbb{R}^N)^{\otimes p}$, that is $\cE^N_1$ are real vectors, $\cE^N_2$ are real matrices, and, for $p \geq 3$, $\cE^N_p$ are real tensors with $p$ legs of dimension $N$. For $T \in \cE^N_p$, we write $T = (T_{i})_{i \in \INT{N}^p} \in \cE^N_p$. The action of $\cM$ on $\cE^N = \sqcup \cE^N_p$ is defined for $\BB \in\cM_q$, with $\partial = (e_1,\ldots,e_q) \subset \vec E(\BB)$ being its boundary, as
\begin{equation}\label{eq:tracem}
\BB ( (T_v)_{ v\in V(\BB)} )_{i_{\partial}} = \frac{1}{N^{\gamma(\BB)}} \sum_{i \in \INT{N}^{E(\BB)}} \prod_{v \in V(\BB)} (T_v)_{i_{\partial v}}.
\end{equation}
Note that we normalize by the number of connected components for the trace maps ($q=0$), and we do not normalize otherwise ($q\geq 1$). 
\par For $\BB \in \cM_0$, the application $\BB : (T_v)_{v \in V} \to \dR$ is multi-linear and importantly, it is orthogonal invariant in the following sense. If $T \in \cE^N_p$ and $U \in \cE_2^N$ is an orthogonal matrix, define $T\cdot U^p \in \cE^N_p$ as the contraction of each leg of $T$ by $U$, that is for $j \in \INT{N}^p$
\begin{equation}\label{eq:defTU}
    (T\cdot U^p)_j = \sum_{i \in \INT{N}^p} T_i \prod_{k=1}^p U_{j_k i_k}.
\end{equation}
In other words, we have $T\cdot U^p = \BB ( (T,U,\ldots ,U) ) $ where $\BB$ is an elongated star map, with $T$ in the middle and $U$ on each branch, the second neighboring edge of $U$ being attached to $T$. If $M \in \cE^N_2$, then $M \cdot U^2 = U M U^{\intercal}$. Then, it is straightforward to check that for any orthogonal matrix $U$ and any trace maps,
$$
\BB ( (T_v \cdot U^{p_v} )_{ v\in V} ) = \BB ( (T_v)_{ v\in V} ),
$$
where $p_v$ is the degree of $v$ (and the order of $T_v$). The trace maps $\BB$ form a complete family of orthogonal invariant multi-linear application.  
They are the natural generalization of trace for matrices so we call them the {\em trace invariants}. We refer to \cite{zbMATH06638014,kunisky2024tensorcumulantsstatisticalinference} for an introduction on these trace invariants (for tensors of even order, it is also possible to define maps which are unitary invariant). We call the {\em distribution} of a collection of tensors the collection of trace invariants for $\BB$ connected and adapted to the collection.

\subsection{Non-crossing poset} 
\par Let $m$ be an even integer and $\pi \in S_m$. We consider $\cM_{\pi} \subset \cM_0$ the set of maps $\BB=(\pi,\alpha)$ for some $\alpha \in S_m$. We construct $\cG_{\pi}$ the graph on $\cM_{\pi}$ where two maps $\BB=(\pi,\alpha)$ and $\BB'=(\pi,\alpha')$ are connected by an edge if $\alpha \alpha'$ is a product of two transpositions. We say that $\BB$ and $\BB'$ differ by a switch. Moreover, we say that $\BB' < \BB$ if they are connected by an edge and $\gamma(\BB')=\gamma(\BB)+1$. We extend the relation $<$ by transitivity giving to $\cM_{\pi}$ the structure of a partially ordered set. We call it by analogy the non-crossing poset as we retrieve the usual one when all vertices of the maps considered have valence $2$. Also, we say that $\BB$ is minimal if there is no $\BB'<\BB$. When all the degrees are even, there is a unique minimal map smaller than a given map as the graph is Eulerian, but it is no more the case when odd degrees appear. That is why we were not able to prove the characterization of freeness by the free cumulants in the odd case. We may probably define the free cumulants of $\BB'\leq \BB$ with a multiplicity being the number of minimal $\BB''$ such that $\BB''\leq \BB' \leq \BB$.
\begin{figure}[h]
    \centering
    \begin{tikzpicture}[scale=0.5]
    \filldraw[gray] (1.5,0) circle (3pt);
    \filldraw[gray] (0,-2) circle (3pt);
    \filldraw[gray] (3,-2) circle (3pt);
    \draw (1.5,0) .. controls (0.5,1.5) and (2.5,1.5) .. (1.5,0);
    \draw (1.5,0) -- (0,-2);
    \draw (0,-2) -- (3,-2) ;
    \draw (1.5,0) -- (3,-2) ;

    \draw (1.5,1) node {\footnotesize 1};
    \draw (0.75,-1) node {\footnotesize 2};
    \draw (2.25,-1) node {\footnotesize 3};
    \draw (1.5,-2) node {\footnotesize 4};

    \draw (1.5,-3.5) node[right] {\footnotesize 2,3};
    \draw (-1.5,-3.5) node[left] {\footnotesize 3,4};
    \draw (4.5,-3.5) node[right] {\footnotesize 2,4};
    
    \draw [very thick] (1.5,-3) -- (1.5,-4);
    \draw [very thick] (-0.5,-3) -- (-2.5,-4);
    \draw [very thick] (3.5,-3) -- (5.5,-4);

    \filldraw[gray] (-4.5,-5.5) circle (3pt);
    \filldraw[gray] (-3,-7.5) circle (3pt);
    \filldraw[gray] (-6,-7.5) circle (3pt);
    \draw (-4.5,-5.5) .. controls (-5.5,-4) and (-3.5,-4) .. (-4.5,-5.5);
    \draw (-4.5,-5.5) -- (-6,-7.5);
    \draw (-4.5,-5.5) .. controls (-4.5,-6.5) and (-5,-7.5) .. (-6,-7.5) ;
    \draw (-3,-7.5) .. controls (-4,-6) and (-2,-6) .. (-3,-7.5);

    \filldraw[gray] (1.5,-5.5) circle (3pt);
    \filldraw[gray] (0,-7.5) circle (3pt);
    \filldraw[gray] (3,-7.5) circle (3pt);
    \draw (1.5,-5.5) .. controls (0.5,-4) and (2.5,-4) .. (1.5,-5.5);
    \draw (1.5,-5.5) .. controls (0.5,-7) and (2.5,-7) .. (1.5,-5.5);
    \draw (0,-7.5) -- (3,-7.5);
    \draw (0,-7.5) .. controls (1,-6.8) and (2,-6.8) .. (3,-7.5);

    \filldraw[gray] (7.5,-5.5) circle (3pt);
    \filldraw[gray] (9,-7.5) circle (3pt);
    \filldraw[gray] (6,-7.5) circle (3pt);
    \draw (7.5,-5.5) .. controls (6.5,-4) and (8.5,-4) .. (7.5,-5.5);
    \draw (7.5,-5.5) -- (9,-7.5);
    \draw (7.5,-5.5) .. controls (7.5,-6.5) and (8,-7.5) .. (9,-7.5) ;
    \draw (6,-7.5) .. controls (5,-6) and (7,-6) .. (6,-7.5);

    \draw [very thick] (-2.5,-8.5) -- (-0.5,-9.5);
    \draw [very thick] (5.5,-8.5) -- (3.5,-9.5);
    \draw [very thick] (1.5,-8.5) -- (1.5,-9.5);

    \filldraw[gray] (1.5,-11) circle (3pt);
    \filldraw[gray] (0,-13) circle (3pt);
    \filldraw[gray] (3,-13) circle (3pt);
    \draw (1.5,-11) .. controls (0.5,-9.5) and (2.5,-9.5) .. (1.5,-11);
    \draw (1.5,-11) .. controls (0.5,-12.5) and (2.5,-12.5) .. (1.5,-11);
    \draw (0,-13) .. controls (-1,-11.5) and (1,-11.5) .. (0,-13);
    \draw (3,-13) .. controls (2,-11.5) and (4,-11.5) .. (3,-13);

    \draw[->] [dashed] (-9,-13) -- (-9,2);
    \draw (-9.5,-5) node[left, scale=1, rotate=90] {$\leq$} ;
    \draw (12.5,0) node[right] {$\mathcal{G}_{\pi}$} ;
    
\end{tikzpicture}
    \caption{Poset $\cP_{\pi}$}
    \label{fig:poset}
\end{figure}
\par There is an associated notion of freeness defined in \cite{bonnin2024freenesstensors}. Importantly, if two tensors $T_1, T_2$ of possibly different orders are freely independent then the distribution of $(T_1,T_2)$ is characterized by the individual distributions of $T_1$ and $T_2$. In particular, we mention that we proved the asymptotic freeness for unitarily invariant families, and in particular Wigner tensors.

\subsection{Moments of a tensor}
If $T$ is a tensor of order $p$ and $\BB$ a $p$-regular trace map, we may define the trace invariant
\begin{equation}
\BB(T,\ldots,T)  =  \frac{1}{N^{\gamma(\BB)}} \sum_{i \in \INT{N}^{E}} \prod_{v \in V} T_{i_{\partial v}}.
\end{equation}
We denote $\BB(T)$ instead of $\BB(T,\cdots,T)$ for ease of notations, and the distribution of $T$ is then the collection of the $\BB(T)$ for all $p$-regular trace map.
\par We remind that $\cB_n$ is the set of connected rooted $p$-regular combinatorial maps in $\cM_0$ with $n$ (unlabeled) vertices.
\begin{definition}[Moments of a tensor]
    For $T$ a tensor of order $p$, we call 
    $$ m_n (T) := \sum_{\BB \in \cB_n} \BB(T) $$
    the $n$-th moment of $T$. By convention we set $m_0(T)=1$. 
\end{definition}
Finally, we define the generating function of these moments, for $z \in \dC$, 
$$ M_{T}(z) := \sum_{n\geq 0} m_n(T) z^n,$$
as a formal power series in $\dC\INT{z}$.

\subsection{Free cumulants}
Given our poset we can naturally define the associated free cumulants. For $\BB \in \cM_0$, the free cumulant of $\BB$ is the unique multilinear application such that
$$ \BB((T_v)_v) = \sum_{\BB'\leq \BB} \kappa_{\BB'}((T_v)_v).$$
It is well defined by Moebius inversion. We refer to the precedent paper for all the properties of this application and their proofs. In particular note that the morphism property still holds, that is it is multiplicative on the connected components.  The crucial point is that we proved that two even families of tensors of possibly different orders are free if and only if the mixed cumulants vanish. 
\par As for the moments, we may then define the free cumulants of a given tensor $T \in \cE_p$.
\begin{definition}[Free cumulants of a tensor]
    For $T$ a tensor of order $p$, we call 
    $$ \kappa_n (T) := \sum_{\BB \in \cB_n^{(p)}} \kappa_{\BB}(T) $$
    the $n$-th free cumulant of $T$. By convention we set $\kappa_0(T)=1$.
\end{definition}
The generating function of these free cumulants $C_T$ is now defined as the formal power series, 
$$ C_{T}(z) := \sum_{n\geq 0} \kappa_n(T) z^n.$$

\subsection{Moment-cumulant formula}
We fix $p\geq 2$ and $T \in \cE_p$. We are going to prove Theorem \ref{thm:MC} which gives a key to compute the moments of the sum of two freely independent tensors. That states that we have the functional relation in $\dC \INT{z}$,
    $$ M_T(z)=C_T(zM_T(z)^{p/2}). $$
The proof of this identity depends on the following Lemma.
\begin{lemma}\label{lem:rel_mc}
For any $n\geq 1$,
    \begin{equation}\label{relation_mc}
    m_n(T) = \sum_{s=1 \atop sp \text{ even}}^n \kappa_{s}(T) \sum_{i_1,\cdots,i_{sp/2} \in \{ 0,\cdots,n-s\} \atop s+i_1+\cdots+i_{sp/2}=n } \prod_{j=1}^{sp/2} m_{i_j}(T).
\end{equation}
\end{lemma}
\begin{proof}[Proof of Lemma \ref{lem:rel_mc}]
Fix $n\geq 1$, a map $\BB\in\cB_n$, and a map $\BB'\leq\BB$. Let $\BB_0$ be the connected component of $\BB'$ containing the root vertex, and let $s$ be its number of vertices. Necessarily $sp$ is even, because $\BB_0$ is $p$-regular, and $\BB_0$ has $sp/2$ edges. By definition of the poset, the remaining connected components of $\BB'$ can be grouped according to the edge of $\BB_0$ through which they are reattached when the switches reconstruct $\BB$. This gives maps $\BB'_1,\ldots,\BB'_{sp/2}$, possibly empty and not necessarily connected. Therefore, using the morphism property of $\kappa_{\BB}$, 
$$ \kappa_{\BB'}(T) = \kappa_{\BB_0}(T) \kappa_{\BB'_1}(T) \ldots \kappa_{\BB'_{sp/2}}(T). $$
If we note $i_j$ the number of vertices in $\BB'_j$, we thus have proved that
\begin{align*}
    m_n(T) &= \sum_{\BB \in \cB_n} \sum_{\BB'\leq \BB} \kappa_{\BB'}(T) \\
    &= \sum_{s=1 \atop sp \text{ even}}^n \sum_{\BB_0 \in \cB_s} \kappa_{\BB_0}(T) \sum_{i_1,\cdots,i_{sp/2} \in \{ 0,\cdots,n-s\} \atop s+i_1+\cdots+i_{sp/2}=n } \prod_{j=1}^{sp/2} \sum_{\BB_j \in \cB_{i_j}} \sum_{\BB'_j \leq \BB_j} \kappa_{\BB'_j}(T) \\
    &= \sum_{s=1 \atop sp \text{ even}}^n \kappa_{s}(T) \sum_{i_1,\cdots,i_{sp/2} \in \{ 0,\cdots,n-s\} \atop s+i_1+\cdots+i_{sp/2}=n } \prod_{j=1}^{sp/2} m_{i_j}(T).
\end{align*}    
The Lemma is proved.
\end{proof}
\begin{proof}[Proof of Theorem \ref{thm:MC}]
    We inject \eqref{relation_mc} in the expression of $M_T(z)$,
\begin{align*}
    M_T(z) &= 1 + \sum_{n=1}^{\infty}  \sum_{s=1 \atop sp \text{ even}}^n \kappa_{s}(T) \sum_{i_1,\cdots,i_{sp/2} \in \{ 0,\cdots,n-s\} \atop s+i_1+\cdots+i_{sp/2}=n } \prod_{j=1}^{sp/2} m_{i_j}(T) z^n \\
    &= \sum_{s\geq 0 \atop sp \text{ even}} \kappa_{s}(T) z^s \sum_{i_1,\cdots,i_{sp/2}} \prod_{j=1}^{sp/2} (m_{i_j}(T) z^{i_j}) \\
    &= \sum_{s\geq 0 \atop sp \text{ even}} \kappa_{s}(T) z^s M_T(z)^{sp/2}.
\end{align*}
For even $p$, this is exactly $C_T(zM_T(z)^{p/2})$. For odd $p$, all odd cumulants vanish (immediate by Moebius inversion, since the odd moments vanish) and the last equation becomes
\[ M_T(z) = \sum_{s\geq 0 } \kappa_{2s}(T) z^{2s} M_T(z)^{sp}. \]
This proves the theorem.
\end{proof}


\section{Higher order laws}\label{sec:laws}

\subsection{Wigner and Wishart tensor}\label{sec:wi}

\paragraph{Symmetric tensors and symmetrization.} For a given $p \geq 1$, the symmetric group $\mathrm{S}_p$ acts on $\INT{N}^p$ by permutation of indices: for $i \in \INT{N}^p$ and $\sigma \in \mathrm{S}_p$, $i_\sigma =  (i_{\sigma(1)},\ldots,i_{\sigma(p)})$. For $i,j \in \INT{N}^p$, we say $i \stackrel{p}{\sim} j$ if $i = j_{\sigma} $ for some $\sigma \in \mathrm{S}_p$.
We say that a tensor is symmetric $T_i = T_j$ as soon as $i \stackrel{p}{\sim} j$, that is $T^{\sigma}=T$ for all $\sigma$, where $T^{\sigma}_i=T_{i_\sigma}$.

\begin{figure}[h!]
    \centering
    \begin{tikzpicture}[scale=0.5]
    \filldraw[gray] (0,0) circle (5pt);
    \draw (0,0) node[above right] {$T$};
    \draw (0,0) -- (0,2) ;
    \draw (0,0) -- (-1.732,-1) ;
    \draw (0,0) -- (1.732,-1) ;

    \draw (0,1) node {\footnotesize 1};
    \draw (-0.866,-0.5) node {\footnotesize 3};
    \draw (0.866,-0.5) node {\footnotesize 2};

    \draw (5,1) node {\footnotesize $\sigma = (1,2)(3)$};
    \draw (5,0) node {$\longrightarrow$};

    \filldraw[gray] (10,0) circle (5pt);
    \draw (10,0) node[above right] {$T^{\sigma}$};
    \draw (10,0) -- (10,2) ;
    \draw (10,0) -- (8.268,-1) ;
    \draw (10,0) -- (11.732,-1) ;

    \draw (10,1) node {\footnotesize 2};
    \draw (9.132,-0.5) node {\footnotesize 3};
    \draw (10.866,-0.5) node {\footnotesize 1};

\end{tikzpicture}
    \caption{$T$ and $T^{\sigma}$.}
\end{figure}
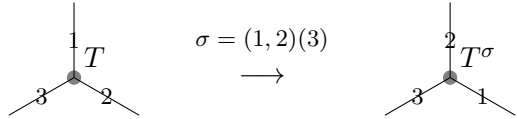

\par If $p_1,p_2$ are two integers, we denote $\cP_{p_1,p_2}$ the set of application $\pi : \INT{p_1+p_2} \mapsto \{1,2 \} \times \INT{p_1+p_2}$ such that $\vert \pi_1^{-1}(1) \vert =p_1, \vert \pi_1^{-1}(2) \vert =p_2$ and $\pi_2$ is injective. We write $\pi^{-1}(1)$ the ordered $p_1$-tuple of elements in $\pi_1^{-1}(1)$ ordered by their image by $\pi$, that is $\pi^{-1}(1)=(x_1,\ldots,x_{p_1})$ such that $\pi(x_j)=(1,i_j)$ where $i_1<i_2<\ldots<i_{p_1}$. We do the same for $\pi^{-1}(2)$. Now if $T_1,T_2$ are tensors of order $p_1,p_2$, we define the symmetrization of $T_1,T_2$ as
$$ T_1 \underline{\otimes} T_2 := \frac{1}{\vert \cP_{p_1,p_2} \vert} \sum_{\pi \in \cP_{p_1,p_2}} (T_1)_{\pi^{-1}(1)} (T_2)_{\pi^{-1}(2)}. $$
In the particular case where $T_1=T_2=T$, we write $T^{\underline{\otimes} 2}$.

\paragraph{Wigner tensors.} We consider $X= (X_i)_{i \in \INT{N}^p} \in (\dR^N)^{\otimes p}$ such that $X_i = X_j$ if $i \stackrel{p}{\sim} j$ and the random variables $(X_i)_{i \INT{N}^p / \stackrel{p}{\sim}}$ are  independent, with
\begin{equation}\label{eq:varX}
 \dE X_i = 0 \quad \AND \quad \dE X_i^2 = \frac{p}{\cP_i},
\end{equation}
where $\cP_i$ is the number of elements in the equivalence class of $i$. We assume moreover that all the moments are bounded, that is
\begin{enumerate}
    \item [(X1)] For all integers $k \geq 2$, there exists a constant $c(k) > 0$ such that for all integers $N\geq 1$ and $i \in \INT{N}^p$:\begin{equation*}\label{eq:Xkbd} \dE |X_i|^k  \leq c(k). \end{equation*}
\end{enumerate}
Remark that we can equivalently write $\dE X_i^2 = \frac{1}{(p-1)!}\prod_{j=1}^N c_j(i)!,$
where $c_j(i)$ is the number of occurrences of $j$ in $i$. 
We also assume that for each integer $N$, the law of $X_i$ depends only on the equivalence class of $i \in \INT{N}^p$ with respect to the action of $\mathrm{S}_N$. The law of $X_i$ may depend on $N$.
\par The main example is the Gaussian Orthogonal Tensor Ensemble (GOTE) where $X$ is Gaussian. 
We define the normalized Wigner tensor as 
$$
\mathbf{W}_N  := \frac{X}{N^{\frac{p-1}{2}}}.
$$
(Note that for vectors, $p=1$, there is no scaling). The random tensor $\mathbf{W}^N$ is the tensor analog of the real Wigner matrices.
\begin{remark}
    For the convergence results which follow, we could assume that the entries have all variance $\frac{1}{(p-1)!}$ because as in the matrix case the terms of leading order in the proof only make appear the entries outside any diagonal, that is when $i_1,\ldots,i_p$ are distinct and then $\frac{p}{\mathcal{P}_i}=\frac{1}{(p-1)!}$.
\end{remark}


\paragraph{Wishart tensors.} We define another main type of real symmetric random tensors. We begin with the even case which is more natural. Let $p\geq 2$ be an even integer. Let $\{ x_{i,j_1,\ldots,j_{p/2}} : i,j \geq 1 \}$ an array of {\em i.i.d.} real random variables with \begin{equation}\label{eq:varx}
 \dE x_{i,j_1,\ldots,j_{p/2}} = 0 \quad \AND \quad \dE x_{i,j_1,\ldots,j_{p/2}}^2 = \frac{1}{\eta_p},
\end{equation}
where $\eta_p:=[(p/2)!]^{2/p}$. 
We assume moreover that all the moments are bounded, that is 
\begin{enumerate}
    \item [(X1')] For all integers $k \geq 2$, there exists a constant $c'(k) > 0$ such that for all integers $N\geq 1$ and $(i,j_1,\ldots,j_{p/2}) \in \INT{N}^{1+p/2}$:\begin{equation*}\label{eq:xkbd} \dE |x_{i,j_1,\ldots,j_{p/2}}|^k  \leq c'(k). \end{equation*}
\end{enumerate}
Then write $x_l^N := (x_{l,j_1,\ldots,j_{p/2}},\ldots,x_{l,j_1,\ldots,j_{p/2}})_{1 \leq j_1,\ldots,j_{p/2} \leq N} \in \cE_{p/2}^N$. We will in the following forget the $N$ in this notation for lightness. For $p\geq 2$, the tensor 
$$x_l^{\underline{\otimes} 2} \in (\dR^N)^{\otimes p}$$
is a symmetric tensor introduced in the previous paragraph.

Then we define the Wishart tensor of size $k$ as
$$
\mathcal{W}_N^k := \frac{x_1^{\underline{\otimes} 2} + \ldots+x_k^{\underline{\otimes} 2}}{N^{\frac{p}{2}}}.
$$
Now assume that $k=k_N$ goes to infinity with $N$. More precisely, we are going to assume that $\frac{k_N}{N} \rightarrow t \in (0, \infty)$. The random tensor $\mathcal{W}^N$ is a Wishart matrix in the particular case $p=2$. Another tensor analog of the real Wishart matrices could be a sum of rank one tensors $x_1^{\otimes p}+\ldots+x_k^{\otimes p}$, this model is not considered here.

\par If $p$ is odd we choose $p_1,p_2\geq 1$ such that $p=p_1+p_2$, for instance $p_1=\frac{p+1}{2}=p_2+1$. Let $\{ x_{i,j_1,\ldots,j_{p_1}} : i,j \geq 1 \}$ and $\{ y_{i,j_1,\ldots,j_{p_2}} : i,j \geq 1 \}$ two arrays of {\em i.i.d.} real random variables with \begin{equation}
 \dE x_{i,j_1,\ldots,j_{p_1}} = \dE y_{i,j_1,\ldots,j_{p_2}} = 0 \quad \AND \quad \dE x_{i,j_1,\ldots,j_{p_1}}^2 = \frac{1}{[p_1!]^{1/p}}, \dE y_{i,j_1,\ldots,j_{p_2}}^2 = \frac{1}{[p_2!]^{1/p}}.
\end{equation}
We denote $\eta_p:=[p_1!p_2!]^{1/p}$. We then construct the Wishart tensor exactly as before 
$$\mathcal{W}_N^k := \frac{x_1 \underline{\otimes} y_1 + \ldots+x_k \underline{\otimes} y_k}{N^{\frac{p}{2}}}.$$

\begin{remark}
    In both Wigner and Wishart cases, the convergence of the moments that we are going to prove in the following still holds if there are only $p$ finite moments, and not necessarily all of them. We may then weaken assumptions $(X1)$ and $(X1')$. To prove that, it is then necessary to define $\widehat x = x \mathbf{1}_{x \leq C}$ and $\widetilde x = \widehat x -\dE \widehat x$ and to control the difference between $\mathbf{W}$ and $\widetilde{\mathbf{W}}$, or $\mathcal{W}$ and $\widetilde{\mathcal{W}}$. We will not treat this case here and we will assume for simplicity that all the moments are bounded. 
\end{remark}

\subsection{Semicircular in high order}\label{sec:SCconv}
The law playing the role of the Gaussian in this free framework is the semicircular of high order, whose even moments are given by the Fuss-Catalan numbers. In the case $p=2$ this law is the semi-circle law and we retrieve usual results of free probability. We define the Fuss-Catalan numbers as
$$ F_p(k) = \frac{1}{pk+1} \binom{pk+1}{k}, $$
which counts the number of $p$-melonic graphs (that is greater than a disjoint union of melons of order $p$ in our poset) with $k$ vertices, or also the number of non-crossing partitions of $\{1,\ldots,(p-1)k \}$ with blocks of size multiple of $p-1$. This set is denoted $NC_{p-1}(k)$ and its cardinality $\#NC_{p-1}(k)=F_p(k)$. A bijection can be given by the partition on the half edges looking on the tree obtained after merging the pair-associated vertices. These numbers already appeared in the literature of free probability for random matrix models \cite{Banica_2011, EDELMAN1980171, zbMATH06794703} and the author would like to understand a potential link between these models and some tensor models in the future.
\begin{definition}[High order semicircular law]
    For $p \in \dN$ we define $\mu_p$ the law having for moments 
    $$ \int \lambda^n d \mu_{p}(\lambda) = F_p(n/2) \mathbf{1}_{n \text{ even}}. $$
    This measure is compactly supported on $[-\sqrt{p^p / (p-1)^{p-1}},\sqrt{p^p / (p-1)^{p-1}}]$.
\end{definition}
Then the moments of a Wigner tensor (and in particular a Gaussian tensor of the GOTE) converges towards the ones of $\mu_p$ when the dimension goes to infinity. It relies on the two following Propositions [Theorem $1$ and Theorem $2$ in \cite{bonnin2024universalitywignerguraulimitrandom}], we will recall some ideas of the proof of the convergence of the moments.
\begin{proposition}\label{prop:momSC}
For all $n \geq 0$ and all $p\geq 2$,
    $$ \mathbb{E} [m_n(\mathbf{W}_N)] = F_p(n/2) \mathbf{1}_{n \text{ even}} + \mathcal{O}(1/N).$$
\end{proposition}
\begin{proposition}
For all $n \geq 0$ and all $p\geq 2$,
    $$\mathrm{Var}[m_n(\mathbf{W}_N)] =\mathcal{O}(1/N^2) .$$

\end{proposition}
\begin{proof}[Ideas of the proof of Proposition \ref{prop:momSC}]
    In some words, the proof consists in classifying the $p$-regular combinatorial maps on $n$ vertices with given indices $j_1,\ldots,j_{np/2}$ associated to each edge into three categories. We define an equivalence relation on the vertices where $x \sim y$ if $x$ and $y$ belong to $p$ edges with the $p$ same indices $j$ associated. We choose one canonical graph in each isomorphism class (two graphs are isomorphic if one is the other after acting by a permutation on $\INT{N}$ for the indices on each edge), and we classify :
\begin{itemize}
    \item[$\bullet$] Category $1$ : each equivalence class contains exactly two vertices (in particular $n$ is even), and the number of distinct indices is maximal. We denote $\mathcal{M}_n$ the set of these graphs.
    \item[$\bullet$] Category $2$ : at least one equivalence class contain one single vertex.
    \item[$\bullet$] Category $3$ : all the other graphs.
\end{itemize}
As the graph is connected the maximal number of distinct indices for a graph in Category $1$ is $1+\frac{n(p-1)}{2}$ ($p$ for the first pair and $p-1$ for each one of the $\frac{n}{2}-1$ others, which is also the consequence of an Euler formula on the reduced tree as there is no cycle). Moreover, there are 
$$ N(N-1)\ldots (N-\frac{n(p-1)}{2}) [(p-1)!]^{n/2} = N^{1+\frac{n(p-1)}{2}} [(p-1)!]^{n/2} (1 + \mathcal{O}(1/N)) $$
graphs in each isomorphism class. We claim that $\mathcal{M}_n$ is exactly the set of the melonic graphs with $n$ vertices with good indices, whose number is given by $F_p(n/2)$. As for a graph in $\mathcal{M}_n$, $\dE X_G = (\dE X_{1,\ldots,p}^2)^{n/2}=\frac{1}{[(p-1)!]^{n/2}}$, then the contribution of these graphs is 
$$ \frac{1}{N} \frac{1}{N^{\frac{n(p-1)}{2}}} \sum_{\mathcal{M}_n} N(N-1)\ldots (N-\frac{n(p-1)}{2}) \times 1 = F_p(n/2) + \mathcal{O}(1/N). $$
Furthermore, by centering of the entries, the graphs in Category $2$ gives a zero contribution. The ones in Category $3$ gives a contribution $\mathcal{O}(1/N)$. That concludes the proof, see \cite{bonnin2024universalitywignerguraulimitrandom} for details.
\end{proof}
To finish with the semicircular we give a proof of Theorem \ref{thm:clt}. Let $(T_i)_{i\geq 1}$ as in our assumptions of the free Central Limit Theorem.
\begin{proof}[Proof of Theorem \ref{thm:clt}]
    Let $s_k:= \frac{1}{\sqrt{k}} \sum_{i=1}^k T_i$. We compute the limit of $\kappa_n(s_k)$ for all $n$ as $k \rightarrow \infty$. By free independence, the mixed cumulants are zero, so we have
    $$ \kappa_n(s_k) = k^{-n/2} \sum_{i=1}^k \sum_{\BB \in \mathcal{B}_n} \kappa_{\BB}(T_i). $$
    Firstly, $\kappa_1(s_k)=m_1(s_k)=0$. Moreover, as $\kappa_2(T_i)=m_2(T_i)- \frac{p}{2} m_1(T_i)^2 =1$, then $\kappa_2(s_k)=\frac{k}{k}=1$. Finally, if $n \geq 3$, then 
    $$ \vert \kappa_n(s_k) \vert \leq \frac{1}{k^{\frac{n}{2}-1}} \mathrm{max}_{\BB \in \mathcal{B}_n} C(\BB) $$
    which goes to zero. That gives the result.
\end{proof}

\subsection{Free Poisson in high order}\label{sec:MPconv}

The other main law of the free world is the free Poisson law, also called the Marčenko-Pastur, which is the limiting measure of a Wishart matrix. We are going to show that the moments of a Wishart tensor converges towards the ones of a higher order free Poisson law. Before describing this law and proving the convergence we may introduce some new notions.
\par For an integer $n$, a partition $\pi$ of $\INT{n}$ is said crossing if there exists $a<x<b<y$ such that $a\overset{\pi}{\sim}b \overset{\pi}{\nsim} x \overset{\pi}{\sim} y$ and it is said non-crossing otherwise. For $q \in \mathbb{Q}, n \in \dN$ we denote $NC_q(n)$ the set of non-crossing partitions of $\INT{qn}$ having blocks of size multiple of $q$ (with $NC_q(n)=\emptyset$ if $qn \notin \dN$). For $q \in \dN$, the cardinal of $NC_q(n)$ is given by the Fuss-Catalan numbers,
$$ \# NC_q(n) = F_{q+1}(n) = \frac{1}{(q+1)n+1} \binom{(q+1)n+1}{n}.$$
Moreover, for $1 \leq b \leq n$ we denote $NC_q^b(n)$ the non-crossing partitions of $\INT{qn}$ having $b$ blocks, whose sizes are multiple of $q$. Then we denote 
$$ F_q^b(n) := \# NC_q^b(n)$$
the cardinal of such partitions. 
\paragraph{Integer case.} For $q\in \dN$, they have been counted, see \cite{EDELMAN1980171}, and their number is given by the Fuss-Narayana numbers, that is 
$$ F_q^b(n) = \frac{1}{b} \binom{n-1}{b-1} \binom{qn}{b-1}, $$
which are a generalization of Narayana numbers. 
\paragraph{Half-integer case.} Now if $q=q'/2$ with $q\notin \dN, q' \in \dN$, then the blocks of a partition having size multiple of $q$ must have a size of length multiple of $q'$ (as $q\dN \cap \dN = q'\dN$). Hence, as $qn \notin \dN$ when $n$ is odd, we can write in the half-integer case
$$ F_q^b(n) = \left\{
    \begin{array}{ll}
        0 & \mbox{if } n \mbox{ odd} \\
        F_{2q}^b(n/2) = \frac{1}{b} \binom{\frac{n}{2}-1}{b-1} \binom{qn}{b-1} & \mbox{if } n \mbox{ even.}
    \end{array}
\right. $$
\begin{definition}[Free Poisson of order $p$]
    For $p \geq 2$ and $t \in (0,\infty)$ we define $\nu_{p,t}$ the law having for moments for $n\geq 1$,
    $$ \int \lambda^n d \nu_{p,t}(\lambda) = \sum_{b=1}^n F_{p/2}^b(n) t^b. $$
    This measure is compactly supported.
\end{definition}
When $p$ is even, this law has for moments the Fuss-Narayana polynomials of parameters $\frac{p}{2},t$ and it is then a so called free Bessel law $\pi_{\frac{p}{2},t}$ as described in \cite{Banica_2011} where it is in particular proved that this measure has a compact support. When $p$ is odd, the odd moments are zero and the even moments are given by the Fuss-Narayana numbers of parameters $p,t$. This law has for image, by the map $z \mapsto z^2$, a free Bessel law of parameters $p,t$. 
\par The existence of this law is then contained in \cite{Banica_2011}.
\paragraph{Case $t=1$.} A relevant point we mention here is the particular case $t=1$. We retrieve in this case the Fuss-Catalan numbers observed for the semicircular, with a different $p$. This means in particular that a free Poisson law of order $2p$ and parameter $t=1$ is the image of  a semicircular law of order $p+1$ by the map $z \mapsto z^2$. In the matrix case, we retrieve that a quarter-circular law (free Poisson - or Marčenko-Pastur as $t=t^{-1}=1$ - of order $2=2\times 1$) is the image by the map $z \mapsto z^2$ of the semicircular (of order $2=1+1$).
\paragraph{Case $t \rightarrow \infty$.} When the parameter $t$ goes to infinity, the reshaped high order free Poisson law $\frac{1}{\sqrt{t}} (f_{p,t} \oplus_p \Delta_{-t})$ tends to the high order semicircular. Here $\Delta_{-t}$ is the distribution with first cumulant equal to $-t$ and others equal to $0$. The proof of this claim is postponed to the following Section (see Remark \ref{rk:t_inf}), after the introduction of the $\oplus_p$ operation.
\par We are now ready to focus on the main result of this section, that is the convergence in probability of the moments of $\mathcal{W}^{p,k}_N$ a $p$-order Wishart tensor such that $k_N / N = t  + \mathcal{O}(N^{-1})$ towards the ones of $f_{p, t}$ as the dimension grows, given by Theorem \ref{thm:Marčenko}. 
The proof of this theorem relies on the moments method. It is indeed sufficient to prove the two following propositions.
\begin{proposition}\label{prop:mom_mp}
    For all $n \geq 1$ and all $p \geq 2$, when $N \rightarrow \infty$, 
    $$ \mathbb{E} [m_n(\mathcal{W}_N)] = \sum_{b=1}^{n} F^b_{p/2}(n) t^b + \mathcal{O}(1/N).$$
\end{proposition}
\begin{proposition}\label{prop:var_mp}
    For all $n \geq 0$ and all $p \geq 2$, when $N \rightarrow \infty$,  
    $$ \mathrm{Var}[m_n(\mathcal{W}_N)] =\mathcal{O}(1/N^2) .$$
\end{proposition}
\begin{remark}\label{rem:MP}
    We want to call attention on a tricky point. The most classical result about convergence of Wishart matrices towards Marčenko-Pastur law treat the case where $N / k_N \rightarrow \tau$ where as mentioned before the limiting Marčenko-Pastur law is equal to the free Poisson law of parameter $t=1/\tau$ dilated by a factor $1/\tau^{p/2}=1/\tau$. The usual Marčenko-Pastur law has also moments given by the Narayana polynomials due to the fact that 
    $$ \tau^n \sum_{b=1}^n \frac{1}{b} \binom{n-1}{b-1} \binom{n}{b-1} \frac{1}{\tau^b} = \sum_{r=0}^{n-1} \frac{1}{n-r} \binom{n-1}{r} \binom{n}{r+1} \tau^r = \sum_{r=0}^{n-1} \frac{1}{r+1} \binom{n-1}{r} \binom{n}{r} \tau^r. $$
    This is no more the case when $p>2$.
\end{remark}

\subsubsection{Proof of Proposition \ref{prop:mom_mp}}
The ideas of this proof follow the ones of Bai and Silverstein in [\cite{BaiSilv}, Section 3] to prove classical convergence towards the Marčenko-Pastur law, adapted in our setting. Note first that the terms due to different choices of $\pi$ in the symmetrization do not interact as the $x_{i,j}$ are chosen {\em i.i.d}. As the computation will not depend on the pairing the normalization by $\cP_{p_1,p_2}$ will disappear with the sum on $\cP_{p_1,p_2}$. Hence we fix the canonical $\pi$ ($p_1$ first indices and $p_2$ last ones) in the rest of the proof and study the asymmetric tensor $(x_1 \otimes y_1 + \ldots+x_k \otimes y_k)/N^{\frac{p}{2}} $.
\par The proof can start. We begin by defining a class of graphs and proving some lemmas about their combinatorics. 
\paragraph{$\Delta$-graphs.} Let $i_1,\ldots i_n$ be $n$ integers (not necessarily distinct) in $\INT{k}$, and $j_1,\ldots,j_{np}$ be $np$ integers (not necessarily distinct) in $\INT{N}$. Denote $b$ the number of distinct elements in $\{ i_1,\ldots i_n \}$ and plot $i_1,\ldots i_n$ on $b$ parallel lines $I_1,\ldots,I_b$ ordered by the first element seen in $i_1,\ldots i_n$. Each line $I_u$ has one vertex and we say that it is a vertex of type $I$. Then, plot $b$ parallel lines $J_1,\ldots,J_b$ such that $J_u$ is between $I_u$ and $I_{u+1}$. Now, recursively for $1\leq v\leq n$, if $i_v$ is on the line $I_u$, then plot the not already plotted $j_{(v-1)p +1},\ldots,j_{vp}$ on $I_u$ and draw $p$ edges from $i_v$ to $j_{(v-1)p +1},\ldots,j_{vp}$ (that can be multiple edges if these $j$ are not distinct). Denote $1+r$ the number of distinct elements in $\{ j_1,\ldots,j_{np} \}$, that is the number of vertices of type $J$. Such a graph always has $np$ edges. A $\Delta$-graph is a graph constructed like that, with the two additional conditions :
\begin{enumerate}
    \item each vertex of type $J$ has valence multiple of $2$ (belongs to $2m$ edges),
    \item the graph is connected.
\end{enumerate}
An example is given in Figure \ref{fig:delta-g}. If $p$ is even, for $1\leq v \leq n$, we call the $(p/2)$-tuples of edges between $i_v$ and $(j_{(v-1)p +1},\ldots,j_{(v-1)p + \frac{p}{2}})$, respectively $(j_{(v-1)p + \frac{p}{2} +1},\ldots,j_{vp})$, the two {\em $v$-multiedges} denoted $e_v(1)$ and $e_v(2)$. We say that two multiedges $e_u(i)$ and $e_v(i')$ are {\em paired} if one can pair them into $p/2$ coincident edges, that is $i_u=i_v$ and there exists $\sigma \in S_{p/2}$ such that for all $1\leq r \leq \frac{p}{2}, j_{(u-1)p + (i-1)\frac{p}{2} + r} = j_{(v-1)p + (i'-1)\frac{p}{2} + \sigma(r)}$. When $p$ is odd, we do the same construction with $e_v(1)$ the $p_1$ first edges and $e_v(2)$ the $p_2$ last ones.  In the even case, it can be $u=v$ and $i=1,i'=2$, but not in the odd case.
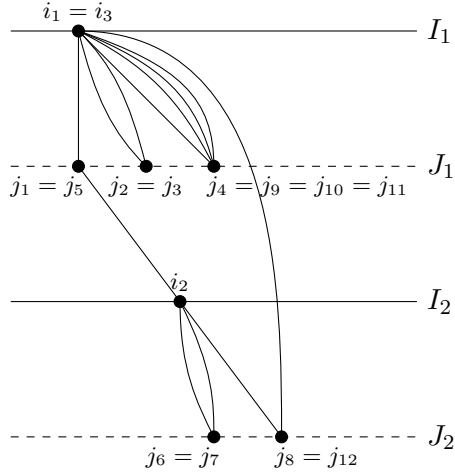
\begin{figure}[h!]
    \centering
    \begin{tikzpicture}[scale=0.45]

    \draw (-2,0) -- (10,0);
    \draw[dashed] (-2,-4) -- (10,-4);
    \draw (-2,-8) -- (10,-8);
    \draw[dashed] (-2,-12) -- (10,-12);

    \draw (10,0) node[right] {$I_1$};
    \draw (10,-4) node[right] {$J_1$};
    \draw (10,-8) node[right] {$I_2$};
    \draw (10,-12) node[right] {$J_2$};

    \filldraw[black] (0,0) circle (5pt);
    \draw (0,0) node[above] {\footnotesize $i_1=i_3$};
    \filldraw[black] (3,-8) circle (5pt);
    \draw (3,-8) node[above] {\footnotesize $i_2$};
    \filldraw[black] (0,-4) circle (5pt);
    \draw (0.5,-4) node[below left] {\footnotesize $j_1=j_5$};
    \filldraw[black] (2,-4) circle (5pt);
    \draw (2,-4) node[below] {\footnotesize $j_2=j_3$};
    \filldraw[black] (4,-4) circle (5pt);
    \draw (3.5,-4) node[below right] {\footnotesize $j_4=j_9=j_{10}=j_{11}$};
    \filldraw[black] (4,-12) circle (5pt);
    \draw (4.5,-12) node[below left] {\footnotesize $j_6=j_7$};
    \filldraw[black] (6,-12) circle (5pt);
    \draw (5.5,-12) node[below right] {\footnotesize $j_8=j_{12}$};

    \draw (0,0) -- (0,-4);
    \draw (0,0) .. controls (0.5,-2) and (1,-3) .. (2,-4);
    \draw (0,0) .. controls (1,-1) and (1.5,-2) .. (2,-4);
    \draw (0,0) -- (4,-4);
    \draw (0,0) .. controls (2,-1) and (3,-2) .. (4,-4);
    \draw (0,0) .. controls (2.5,-1) and (3.5,-2) .. (4,-4);
    \draw (0,0) .. controls (3,-1) and (4,-2) .. (4,-4);

    \draw (0,0) .. controls (6,0) and (6,-8) .. (6,-12);

    \draw (0,-4) -- (3,-8);
    \draw (3,-8) .. controls (4,-10) and (4,-11) .. (4,-12);
    \draw (3,-8) .. controls (3,-10) and (3.5,-11) .. (4,-12);
    \draw (3,-8) -- (6,-12);

\end{tikzpicture}
    \caption{A $\Delta$-graph ($p=4, n=3, b=2, r+1=5$).}
    \label{fig:delta-g}
\end{figure}
\par Two graphs are said isomorphic if there are the same up to a permutation on $(1,\ldots,k)$ and a permutation on $(1,\ldots,N)$. For each isomorphism class, there is only one {\em canonical} graph satisfying $i_1=j_1=1$, $i_{u+1} \leq \mathrm{max}\{ i_1,\ldots,i_u \} +1$ and $j_{u+1} \leq \mathrm{max}\{ j_1,\ldots,j_u \} +1$. The set of canonical $\Delta$-graphs with $b$ vertices of type $I$ and $1+r$ vertices of type $J$ is denoted $\Delta(n,b,r)$. We then classify the $\Delta(n,b,r)$-graphs into $3$ categories :
\begin{itemize}
    \item[$\bullet$] Category 1 : $\Delta(n,b,r)$-graphs in which there is no cycle, each edge coincides with one and only one other edge, and the multiedges are paired (two by two). If we glue the double edges, the resulting graph is a tree with $kp/2$ edges. Hence, $\frac{kp}{2}= b+r$ and we can just denote these graphs the set of $\Delta_1(n,b)$-graphs as $r$ can be forgotten. 
    \item[$\bullet$] Category 2 : $\Delta(n,b,r)$-graphs that contain at least one non-paired multiedge. We denote them the $\Delta_2(n,b,r)$.
    \item[$\bullet$] Category 3 : $\Delta(n,b,r)$-graphs which do not belong to $\Delta_1(n,b)$ or $\Delta_2(n,b,r)$. We denote them the $\Delta_3(n,b,r)$-graphs.
\end{itemize}
An example of $\Delta_1(n,b)$-graph is given in Figure \ref{fig:delta1-g}. 
\begin{figure}[h]
    \centering
    \begin{tikzpicture}[scale=0.45]

    \draw (-2,0) -- (10,0);
    \draw[dashed] (-2,-4) -- (10,-4);
    \draw (-2,-8) -- (10,-8);
    \draw[dashed] (-2,-12) -- (10,-12);

    \draw (10,0) node[right] {$I_1$};
    \draw (10,-4) node[right] {$J_1$};
    \draw (10,-8) node[right] {$I_2$};
    \draw (10,-12) node[right] {$J_2$};

    \filldraw[black] (0,0) circle (5pt);
    \draw (0,0) node[above] {\scriptsize $i_1=i_3$};
    \filldraw[black] (3,-8) circle (5pt);
    \draw (3,-8) node[above right] {\scriptsize $i_2$};
    \filldraw[black] (0,-4) circle (5pt);
    \draw (0.5,-4) node[below left] {\scriptsize $j_1=j_4$};
    \filldraw[black] (2,-4) circle (5pt);
    \draw (2,-4) node[below] {\scriptsize $j_2=j_3$};
    \draw (2,-5) node [below] {\scriptsize $=j_5=j_8$};
    \filldraw[black] (4,-4) circle (5pt);
    \draw (4.3,-4) node[below] {\scriptsize $j_9=j_{10}$};
    \filldraw[black] (4,-12) circle (5pt);
    \draw (4,-12) node[below] {\scriptsize $j_6=j_7$};
    \filldraw[black] (6,-4) circle (5pt);
    \draw (5.5,-4) node[below right] {\scriptsize $j_{11}=j_{12}$};

    \draw (0,0) .. controls (-0.5,-1.5) and (-0.5,-2.5) .. (0,-4);
    \draw (0,0) .. controls (0.5,-1.5) and (0.5,-2.5) .. (0,-4);
    \draw (0,0) .. controls (0.5,-2) and (1,-3) .. (2,-4);
    \draw (0,0) .. controls (1,-1) and (1.5,-2) .. (2,-4);
    \draw (0,0) -- (4,-4);
    \draw (0,0) .. controls (2,-1) and (3,-2) .. (4,-4);
    \draw (0,0) -- (6,-4);
    \draw (0,0) .. controls (3,-1) and (4.5,-2) .. (6,-4);
    \draw (3,-8) .. controls (4,-10) and (4,-11) .. (4,-12);
    \draw (3,-8) .. controls (3,-10) and (3.5,-11) .. (4,-12);
    \draw (3,-8) .. controls (2,-6) and (2,-5) .. (2,-4);
    \draw (3,-8) .. controls (3,-6) and (2.5,-5) .. (2,-4);

\end{tikzpicture}
    \caption{A $\Delta_1(n,b)$-graph.}
    \label{fig:delta1-g}
\end{figure}
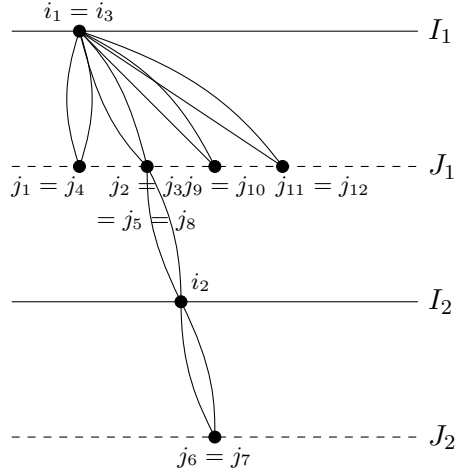
We can now prove some results about their combinatorics.
\begin{lemma}\label{lem:isomo}
    For given $n,b,r$, the number of graphs in the isomorphism class of $G \in \Delta(n,b,r)$ is 
    $$ k(k-1)\ldots(k-b+1) N(N-1)\ldots(N-r) c_{p}^G = k^b N^{r+1} c^G_{p} (1 + \mathcal{O}(N^{-1})).$$
    where $c^G_{p}$ is a constant depending only on the shape of $G$, with in particular if $G \in \Delta_1(n,b)$, 
    $$ c^G_p = \eta_p^{np/2}. $$
\end{lemma}
\begin{proof}
    The factor $\eta_p^{np/2}$ comes from the choice of the pairing of the $j$ in each pair of multiedges. One must distinguish the even and the odd cases. In the even case, we have $c_p^G = [(p/2)!]^n$. In the odd case, $n$ must be even as we cannot pair multiedges from a same vertex of type I, so $c_p^G= [p_1!p_2!]^{n/2}$. In both cases, that is equal to $\eta_p^{np/2}$. The rest is clear.
\end{proof}
\begin{lemma}\label{lem:delta3}
    The total number of non-coincident vertices of a $\Delta_3(n,b,r)$-graph is at most $np/2$.
\end{lemma}
\begin{proof}
    Let $G$ be a graph in $\Delta_3(n,b,r)$. Since $G$ is not in $\Delta_2$, it has no single edge, so the number of non-coincident edges is at most $np/2$. As any $\Delta$-graph is connected, the number of non-coincident vertices is at most the number of non-coincident edges plus one with equality if and only if the non-coincident graph is a tree. But $G$ is not in $\Delta_1$, so if it has exactly $np/2$ non-coincident vertices, then it must contain a cycle. Hence, in all cases the number of non-coincident vertices is at most $np/2$. 
\end{proof}
The last very important result is to count the number of $\Delta_1(n,b)$-graphs.
\begin{lemma}\label{lem:delta1}
    For $n,b$, the number of $\Delta_1(n,b)$-graphs is 
    $$ F^b_{p/2}(n).$$
\end{lemma}
\begin{proof}
    This lemma follows from a bijection between $\Delta_1(n,b)$-graphs and non-crossing partitions of $\INT{np/2}$ with $b$ blocks, whose sizes are multiples of $\frac{p}{2}$. For a $\Delta_1(n,b)$-graph $G$ we glue the coincident edges, it remains $\frac{np}{2}$ edges which we label by their order of visit in the depth-first search starting from $x_1$. Then we associate to this graph the (non-crossing) partition on $\INT{np/2}$ with $b$ blocks where $j_1,j_2$ belong to the same block if and only if the vertex of type $I$ belonging to the edges labeled $j_1$ and $j_2$ is the same. This bijection gives also easily the way to reconstruct a $\Delta_1(n,b)$-graph from a partition in $\mathrm{NC}^b_{p/2}(n)$.
\end{proof}
We are now ready to prove the convergence of the moments.
\begin{proof}[Proof of Proposition \ref{prop:mom_mp}.]
    A first calculation gives :
    \begin{align*}
        \mathbb{E} [m_n(\mathcal{W}_N)] &= \frac{1}{N} \frac{1}{N^{np/2}} \sum_{\BB \in \mathcal{B}_n} \mathbb{E}\BB(x_1^{\underline{\otimes} 2} + \ldots+ x_k^{\underline{\otimes} 2}) \\
        &= \frac{1}{N^{1+np/2}} \sum_{\BB=(V,E) \in \mathcal{B}_n} \sum_{1 \leq i_1,\ldots,i_n \leq k} \sum_{1 \leq j_1,\ldots,j_{np/2} \leq N} \prod_{a=1}^n \mathbb{E}(x_{i_a})_{j_{(\delta v_a)_1}} \ldots (x_{i_a})_{j_{(\delta v_a)_p}}. 
    \end{align*}
Two terms are equal if the corresponding graphs are isomorphic. Therefore by Lemma \ref{lem:isomo},
$$  \mathbb{E} [m_n(\mathcal{W}_N)] = \frac{1}{N^{1+np/2}}\sum_{G \in \Delta(n,b,r)} k(k-1)\ldots(k-b+1) N(N-1)\ldots(N-r) c_p^G \mathbb{E}\mathbf{x}_G,$$
where the summation is over all canonical $\Delta(n,b,r)$-graphs. The summation is into three parts according to the type of the graph $\Delta_a(n,b,r)$ for $a=1,2$ or $3$.
Firstly, if $G$ is a $\Delta_2(n,b,r)$-graph then $\mathbb{E}\mathbf{x}_G=0$ as the $x_i$ are centered and there exists at least one non-paired multiedge. That is,
$$  S_2 = \frac{1}{N^{1+np/2}}\sum_{G \in \Delta_2(n,b,r)} k(k-1)\ldots(k-b+1) N(N-1)\ldots(N-r) c_p^G \mathbb{E}\mathbf{x}_G =0.$$
Secondly, if $G$ is a $\Delta_3(n,b,r)$-graph then $b+r< np/2$ by Lemma \ref{lem:delta3} and as $\mathbb{E}\mathbf{x}_G \leq c'(n)$ then
$$  S_3 = \frac{1}{N^{1+np/2}}\sum_{G \in \Delta_3(n,b,r)} k(k-1)\ldots(k-b+1) N(N-1)\ldots(N-r)c^G_p \mathbb{E}\mathbf{x}_G = \mathcal{O}(N^{-1}).$$
Finally, if $G$ is a $\Delta_1(n,b)$-graph then $\mathbb{E}\mathbf{x}_G=(\mathbb{E}[(x_1)_1^2])^{np/2}=1/\eta_p^{np/2}$. So by Lemma \ref{lem:delta1},
\begin{align*}
    S_1 &= \frac{1}{N^{1+np/2}}\sum_{G \in \Delta_1(n,b)} k(k-1)\ldots(k-b+1) N(N-1)\ldots(N-r) \eta_p^{np/2} \mathbb{E}\mathbf{x}_G \\
    &= \frac{1}{N^{1+np/2}}\sum_{b=1}^n F_{p/2}^b(n) k^{b} N^{1+\frac{np}{2}-b} + \mathcal{O}(N^{-1}) \\
    &= \sum_{b=1}^{n} F_{p/2}^b(n) t_N^b + \mathcal{O}(N^{-1}),
\end{align*}  
where $t_N = k_N / N \rightarrow t$. This concludes the proof.
\end{proof}

\subsubsection{Proof of Proposition \ref{prop:var_mp}}
We only need to show that $\mathrm{Var} [m_n(\mathcal{W}_N)]$ is summable for all fixed $n$.
\begin{proof}[Proof of Proposition \ref{prop:var_mp}]
We now have
$$ \mathrm{Var} [m_n(\mathcal{W}_N)] = \frac{1}{N^2} \frac{1}{N^{np}} \sum_{G_1, G_2 \in \Delta} [\mathbb{E}\mathbf{x}_{G_1} \mathbf{x}_{G_2} - \mathbb{E}\mathbf{x}_{G_2}\mathbb{E}\mathbf{x}_{G_1}] . $$
Firstly, if $G_1$ has no coincident edge with $G_2$, then by independence
$$ \mathbb{E}\mathbf{x}_{G_1} \mathbf{x}_{G_2} - \mathbb{E}\mathbf{x}_{G_2}\mathbb{E}\mathbf{x}_{G_1} =0. $$
Moreover, if $G_1 \cup G_2$ has an overall single edge, then we have also
$$ \mathbb{E}\mathbf{x}_{G_1} \mathbf{x}_{G_2} = \mathbb{E}\mathbf{x}_{G_2}\mathbb{E}\mathbf{x}_{G_1} =0. $$
Hence, we can assume that $G_1$ has a coincident edge with $G_2$ and contains $G =G_1 \cup G_2$ has no single edge. If the graph of noncoincident edges has a cycle, then the noncoincident edges of $G$ are not more than $np$. In the other case, if the graph of noncoincident edges has no cycle, then the number of noncoincident vertices is also not larger than $np$ due to the following reason. At least one edge must have coincidence greater than $3$, because otherwise we find an edge of coincidence $1$ in both $G_1$ and $G_2$, and then a second one as the number of edges is even, so that gives a cycle in $G$. Thus, either this edge has coincidence greater than $4$, either it is equal to $3$ and there is another one of coincidence greater than $3$ since again the number of edges is even. That gives that the number of noncoincident vertices is not larger than $(np-1)+1=np$. Consequently, we get
$$ \mathrm{Var} [m_n(\mathcal{W}_N)] \leq \frac{1}{N^2} 2 c'(n) K_n, $$
$K_n$ depending only on $n$. The proof is complete.
\end{proof}


\section{Tensorial free convolution}\label{sec:conv}

The presentation of this section is inspired by the lecture notes of Roland Speicher \cite{speicher2019lecturenotesfreeprobability}. For this section we fix $p \geq 2$ an even integer.

\subsection{Algebraic free convolution}

Let $a$ be a distribution on the trace maps with moments $(m_n(a))_{n\in \dN}$ and free cumulants  $(\kappa_n(a))_{n\in \dN}$. Recall that $(m_n(a))_{n\geq0}$ is a real sequence with $m_0=1$, we write $M_a(z)=\sum_{n\geq0}m_n(a) z^n$.
The cumulant series $C_a(z)=\sum_{n\geq0}\kappa_n(a)z^n$, with $\kappa_0(m)=1$, is defined uniquely by
\begin{equation} \label{eq:thm1}   M_a(z)=C_a\bigl(zM_a(z)^{p/2}\bigr).
\end{equation}
Indeed, $zM_a(z)^q=z+\mathcal O(z^2)$ has a unique compositional inverse. Conversely, a cumulant series with constant term one determines a unique moment series through the same equation.

\begin{proposition}\label{prop:exp_bound}
    The following statements are equivalent :
    \begin{itemize}
        \item[(i)] the sequence $(m_n(a))_{n\in \dN}$ is exponentially bounded,
        \item[(ii)] the sequence $(\kappa_n(a))_{n\in \dN}$ is exponentially bounded.
    \end{itemize}
\end{proposition}
\begin{proof}
$\boxed{(ii) \Rightarrow (i)}$ Assume that $\kappa_n(a) \leq M^n$ for all $n$. Then for $n \geq 0$,
        $$ \vert m_n(a) \vert \leq \sum_{r=1}^n \sum_{\pi_n = (\pi_{i_1},\ldots,\pi_{i_r}) \in \mathrm{NC}_p^r(n)} \underbrace{\vert \sum_{\BB_1,\ldots,\BB_r \in \mathcal{B}_{i_1} \times \ldots \times \mathcal{B}_{i_r}}  \kappa_{\BB_1}(a) \ldots \kappa_{\BB_r}(a) \vert}_{= |\kappa_{i_1}(a)\ldots \kappa_{i_r}(a)| \leq M^n} \leq F_p(n) M^n \leq (2^p M)^n. $$
        
$\boxed{(i) \Rightarrow (ii)}$ We have by Moebius inversion
$$ \kappa_{\BB}(a) = \sum_{\BB' \leq \BB} \mathrm{Moeb}(\BB', \BB) \BB'(a),$$
where $\mathrm{Moeb}(\BB',\BB)$ is the Moebius function associated to our poset. It is defined as follows: for any $\BB$, $\mathrm{Moeb}(\BB,\BB)=1$ and for $\BB'<\BB$, $\mathrm{Moeb}(\BB',\BB)=-\sum_{\BB'\leq \BB''<\BB} \mathrm{Moeb}(\BB',\BB'')$. Note that if $\BB \in \mathcal{B}_n$ and $\BB'\leq \BB$, then $\vert \mathrm{Moeb}(\BB',\BB) \vert \leq \#NC_{p-1}(n)=F_p(n)$. Moreover, $\mathrm{Moeb}(\BB',\BB)$ only depends on the sequence of switches to get $\BB'$ from $\BB$. When $\BB$ is connected, we have $\mathrm{Moeb}(\BB',\BB)=\mathrm{Moeb}(\pi_n,1_n)$ where $\pi_n$ is the partition of the $n$ vertices of $\BB'$ into connected components and $1_n$ is the partition over $n$ vertices with one single block.
Now assume that $|m_n(a)| \leq M^n$ for all $n$. Then for $n \geq 0$,
$$ \vert \kappa_n(a) \vert \leq \sum_{r=1}^n \sum_{\pi_n = (\pi_{i_1},\ldots,\pi_{i_r}) \in \mathrm{NC}_p^r(n)} \underbrace{\vert \mathrm{Moeb}(\pi_n, 1_n) \vert \vert \sum_{\BB_1,\ldots,\BB_r \in \mathcal{B}_{i_1} \times \ldots \times \mathcal{B}_{i_r}} \BB_1(a) \ldots \BB_r(a) \vert }_{\leq F_p(n) M^n \leq (2^p M)^n } \leq (4^p M)^n. $$
\end{proof}

\begin{definition}[Algebraic tensorial free convolution]
     Let $a$ and $b$ be two distributions on the trace maps with exponentially bounded moments sequences $(m_n(a))$ and $(m_n(b))$. Their algebraic tensorial free convolution is the unique sequence $(m_n(a \oplus_p b))$ whose associated free cumulants satisfy $\kappa_0(a \oplus_p b)=1$ and for $n\geq 1$,
     \[ \kappa_n(a \oplus_p b)= \kappa_n(a) + \kappa_n(b). \]
\end{definition}
Proposition \ref{prop:exp_bound} proves that the resulting moments are again exponentially bounded.
When $p=2$ we retrieve the usual free convolution. 
Let $\delta_0$ be the distribution with value $0$ on every nonempty trace map (and $\delta_0(\emptyset)=1$). The operation $\oplus_p$ has the following properties :
\begin{enumerate}
    \item[(N)] Neutral element : $(m_n(\delta_0 \oplus_p a))=(m_n(a))$,
    \item[(C)] Commutativity : $(m_n(a \oplus_p b))=(m_n(b \oplus_p a))$,
    \item[(A)] Associativity : $(m_n(a \oplus_p (b \oplus_p c)))=(m_n((a \oplus_p b) \oplus_p c))$.
\end{enumerate}

\begin{remark}\label{rem:negative_moments}
    Even if the sequences $(m_n(a))$ and $(m_n(b))$ are the moments of probability measures, it is not necessarily the case of $(m_n(a \oplus_p b))$. The author thanks the anonymous referee for suggesting the following example. Fix $p>2$ even and write $q=p/2\geq 2$. Then expanding \eqref{eq:thm1} through order two gives
    $$ \kappa_1=m_1,\qquad \kappa_2=m_2-qm_1^2.$$
    For $(m_n(a))=(1)_{n\geq 1}$, the first two cumulants are $(1,1-q)$; for $(m_n(b))=((-1)^n)_{n\geq 1}$, they are $(-1,1-q)$. Their algebraic convolution therefore has
    $$\kappa_1=0,\qquad \kappa_2=2(1-q)=2-p.$$
    Because the first cumulant vanishes, the candidate second moment is also $m_2=2-p<0$. 
\end{remark}

Another remark has to be done, $\delta_0$ is a particular case ($t=0$) of the distribution $t . 1_{p}$ having for free cumulants
$$ \kappa_{\BB}(t . 1_{p}) = \frac{t}{(p-1)(p-3)\ldots 1} \mbox{ if } \BB \mbox{ is a bouquet map and } 0 \mbox{ otherwise.}$$
It is the identity appearing in \cite{bonnin2024freenesstensors}, symmetrized, and multiplied by a factor $t$. It has an associated probability measure on $\dR$, denoted $\Delta_t$ having for free cumulants
$$ \kappa_n(\Delta_t)=t \mbox{ if } n=1 \mbox{ and } 0 \mbox{ if } n \geq 2,$$
and for moments 
$$ m_n(\Delta_t) = F_{p/2}(n) t^n.$$
This follows from Theorem \ref{thm:MC}, since the functional relation satisfied by the $(p/2)$-Fuss-Catalan generating function is $f(z)=1+zf(z)^{p/2}$. This means in particular that the identity $1_4$ of order $4$ has for moments the Catalan numbers. Importantly, these distributions $t . 1_{p}$ are {\em free from anything} as we proved in [\cite{bonnin2024freenesstensors}, Lemma 7].

\subsection{Basic examples for the semicircular and free Poisson}
Firstly, we note that the analytic moment-cumulant formula given by Theorem \ref{thm:MC} allows to retrieve the sequence of free cumulants from the sequence of moments and reciprocally with less effort than the combinatorial moment-cumulant formula on the poset. That is what gives the two following lemmas. 
\begin{lemma}
    For $n \geq 1$, the free cumulants of the semicircular law of order $p$ are given by 
    $$ \kappa_n = \mathbf{1}_{n=2}. $$
\end{lemma}
\begin{proof}
    Assume that $\pi$ is such that $C_{\pi}(z)=1+z^2$. Then,
    $$M_{\pi}(z)= 1+ z^2 M_{\pi}(z)^p.$$
    Hence a simple recursive argument gives that the odd moments are zero, so $M_{\pi}(z)=f(z^2)$ and we retrieve the functional equation of the Fuss-Catalan numbers $f(z)=1+zf(z)^p$. That gives that $\pi=sc_p$.
\end{proof}
\begin{lemma}
    For $n\geq 1$, the free cumulants of a free Poisson law of order $p$ and parameter $t$ are given by
    $$ \kappa_n =t $$
    if $p$ is even, and 
    $$ \kappa_n =\mathbf{1}_{n \text{ even }} t$$
    if $p$ is odd.
\end{lemma}
\begin{proof}
    We treat the even case, the odd case follows. Assume that $\pi$ has all non trivial free cumulants equal to $t$. Then for $\vert z \vert <1$,
    $$C_{\pi}(z)=1+ \frac{zt}{1-z}.$$
    That gives the following relation for $M_{\pi}(z)$
    $$ M_{\pi}(z) = 1 + zM_{\pi}^{p/2}(z) (M_{\pi}(z) +t-1),$$
    which is the equation satisfied by the generative function of the Fuss-Narayana numbers, see \cite{Banica_2011}. This gives $\pi=f_{p,t}$.
\end{proof}
\begin{remark}\label{rk:t_inf}
    When the parameter $t$ goes to infinity, the reshaped free Poisson law $\frac{1}{\sqrt{t}} (f_{p,t} \oplus_p \Delta_{-t})$ tends to a semicircular.  
    Indeed, since $1_p$ is free from anything as we mentioned just before, 
    $$ \kappa_n(f_{p,t} \oplus (-t) . 1_p)=0 \mbox{ if } n=1 \mbox{ and } t \mbox{ if } n \geq 2. $$
    Then, we have that 
    $$ \kappa_2(\frac{1}{\sqrt{t}}(f_{p,t} \oplus_p \Delta_{-t}))=1, $$
    and for all $n\geq 3$, when $t \rightarrow \infty$,
    $$ \kappa_n(\frac{1}{\sqrt{t}}(f_{p,t} \oplus_p \Delta_{-t})) = t^{1-\frac{n}{2}} \rightarrow 0. $$
    This proves the result.
\end{remark}
We give in the following the basic examples of tensorial free convolution for high order semicircular and free Poisson laws.
\begin{lemma}
    The free convolution of two freely independent semicirculars of order $p$ is a semicircular of order $p$ dilated by a factor $\sqrt{2}$. That is,
    $$ m_n (sc_p \oplus_p sc_p) = m_n( sc_p^{(\sqrt{2})}).$$
\end{lemma}
\begin{proof}
    One can just write
    $$ m(z):= M_{sc_p \oplus_p sc_p}(z) = C_{sc_p}(zm(z)^{p/2}) + C_{sc_p}(zm(z)^{p/2}) -1 = 1+ 2z^2 m(z)^p. $$
    That gives the result.
\end{proof}
\begin{lemma}
    The free convolution of two freely independent free Poisson of order $p$ and parameters $t$ and $t'$ is a free Poisson of order $p$ and parameter $t+t'$. That is,
    $$ m_n(f_{p,t} \oplus_p f_{p,t'}) = m_n(f_{p,t+t'}). $$
\end{lemma}
\begin{proof}
    We have for $\vert z \vert <1$,
    $$ m(z):= M_{f_{p,t} \oplus_p f_{p,t'}}(z) = C_{f_{p,t}}(zm(z)^{p/2}) + C_{f_{p,t'}}(zm(z)^{p/2}) -1 = 1+ \sum_{n \geq 1}(t+t') (zm(z)^{p/2})^n. $$
    Again, we get the result.
\end{proof}

\subsection{$R$-transform and subordination}
\begin{definition}[$R$-transform]
    For an exponentially bounded moment sequence $(m_n(a))$, we define its $R$-transform as the formal power series
    $$ R_{a}(z)= \frac{C_{a}(z)-1}{z} = \sum_{n=1}^{\infty} \kappa_n(a) z^{n-1}. $$
\end{definition}
\noindent By construction, we have for $\vert z \vert$ sufficiently small,
\[ R_{a \oplus_p b}(z) = R_{a}(z) + R_{b}(z).\]
For ease of notations we omit the subscript and set
\[ G(z):=\frac{1}{z}M\left(\frac{1}{z}\right)^{p/2} , \qquad K(z):=\frac{C_T(z)^{p/2}}{z}. \]
Theorem \ref{thm:MC} implies 
$$ C(G(z))^{p/2} = zG(z).$$
and therefore we obtain 
\begin{equation}\label{eq:KG}
    K(G(z))=z .
\end{equation}
Since $G(z)=z^{-1}+\mathcal{O}(z^{-2})$ and $K(z)=z^{-1}+ \mathcal{O}(1)$, the two Laurent series are compositional inverses:
\begin{equation}\label{eq:GK}
    G(K(z))=z.
\end{equation}
Since $K(z)$ has a pole $\frac{1}{z}$ we may split this off and write 
$$ K(z) = \frac{1}{z} + Q(z) \quad \mbox{where} \quad Q(z):= \sum_{n=1}^{\infty} \sum_{i_1,\ldots, i_{p/2} \atop i_1+\ldots+i_{p/2}=n} \kappa_{i_1}\ldots \kappa_{i_{p/2}} z^{n-1}.$$
We are now able to prove the subordination property.
\begin{proof}[Proof of Theorem \ref{thm:subordination}]
    Let $a, b$ be two freely independent distributions. For $|z|$ sufficiently large,
\begin{align*}
    G_{a \oplus_p b}(z) &= G_{a}\left( K_{a}\left[G_{a \oplus_p b}(z)\right]\right) \\
    &= G_{a}\left( C_{a}\left[G_{a \oplus_p b}(z)\right]^{p/2}/G_{a \oplus_p b}(z)\right) \\
    &= G_{a}\left( \frac{\left( C_{a \oplus_p b}\left[G_{a \oplus_p b}(z)\right] - C_{b}\left[G_{a \oplus_p b}(z)\right]+1 \right)^{p/2}}{G_{a \oplus_p b}(z)}\right) 
\end{align*}
Now rewrite $C_{a \oplus_p b}\left[G_{a \oplus_p b}(z)\right]$ as a $(p/2)$-th root of $zG_{a \oplus_p b}(z)$, and select the unique analytic branch tending to $1$ as $z \to \infty$. This choice being done, we may rewrite
\begin{align*}
        G_{a \oplus_p b}(z) &= G_{a}\left( \left[ z^{2/p}-\frac{ C_{b}\left[G_{a \oplus_p b}(z)\right]-1}{G_{a \oplus_p b}(z)^{2/p}}\right]^{p/2} \right) \\
    &= G_{a}\left( \left[ z^{2/p}-\hat R_{b} \left( G_{a \oplus_p b}(z)\right)\right]^{p/2} \right),
\end{align*}
where $\hat R(z):=zR(z)/z^{2/p}$ with the right branch choice. That is the desired result.
\end{proof}

\begin{remark}\label{rmk:subsc}
    When $b$ is semicircular, note that $R_{b}(z)=z$. The equation simplifies. Remark also that $\left[z^{2/p}-G_{b}(z)^{(2p-2)/p}\right]^{p/2}= 1/G_{b}(z)$ since $M_{b}=1+z^2M_{b}(z)^p$. It is the method used by Capitaine and Donatti-Martin \cite{capitaine2016spectrumdeformedrandommatrices} in the matrix case to exhibit a BBP transition using free probability.
\end{remark}
Finally, when $(m_n(a))$ is no longer exponentially bounded, one should study analytic properties of $G$ (injectivity and existence of an inverse sufficiently far from $0$). We may then define the subordination functions and try to study the general tensorial free convolution.

\bibliographystyle{abbrv}
\bibliography{bib}

@misc{bonnin2024freenesstensors,
      title={Freeness for tensors}, 
      author={Remi Bonnin and Charles Bordenave},
      year={2024},
      eprint={2407.18881},
      archivePrefix={arXiv},
      primaryClass={math.PR},
      url={https://arxiv.org/abs/2407.18881}, 
}

@misc{bonnin2024universalitywignerguraulimitrandom,
      title={Universality of the Wigner-Gurau limit for random tensors}, 
      author={Remi Bonnin},
      year={2024},
      eprint={2404.14144},
      archivePrefix={arXiv},
      primaryClass={math.PR},
      url={https://arxiv.org/abs/2404.14144}, 
}

@misc{kunisky2024tensorcumulantsstatisticalinference,
      title={Tensor cumulants for statistical inference on invariant distributions}, 
      author={Dmitriy Kunisky and Cristopher Moore and Alexander S. Wein},
      year={2024},
      eprint={2404.18735},
      archivePrefix={arXiv},
      primaryClass={math.ST},
      url={https://arxiv.org/abs/2404.18735}, 
}

@misc{gurau2020generalizationwignersemicirclelaw,
      title={On the generalization of the Wigner semicircle law to real symmetric tensors}, 
      author={Razvan Gurau},
      year={2020},
      eprint={2004.02660},
      archivePrefix={arXiv},
      primaryClass={math-ph},
      url={https://arxiv.org/abs/2004.02660}, 
}

@misc{speicher2019lecturenotesfreeprobability,
      title={Lecture Notes on "Free Probability Theory"}, 
      author={Roland Speicher},
      year={2019},
      eprint={1908.08125},
      archivePrefix={arXiv},
      primaryClass={math.OA},
      url={https://arxiv.org/abs/1908.08125}, 
}

@article{Banica_2011,
   title={Free Bessel Laws},
   volume={63},
   ISSN={1496-4279},
   url={http://dx.doi.org/10.4153/CJM-2010-060-6},
   DOI={10.4153/cjm-2010-060-6},
   number={1},
   journal={Canadian Journal of Mathematics},
   publisher={Canadian Mathematical Society},
   author={Banica, T. and Belinschi, S. T. and Capitaine, M. and Collins, B.},
   year={2011},
   month=feb, pages={3–37} }

@Book{BaiSilv,
  author={Bai, Zhidong and Silverstein,  Jack W.},
  title={Spectral Analysis of Large Dimensional Random Matrices},
  Publisher = {Springer Ser. Statist.},
  year={2010},
  pages={XVI, 552},
  doi={10.1007/978-1-4419-0661-8}}

@article{EDELMAN1980171,
title = {Chain enumeration and non-crossing partitions},
journal = {Discrete Mathematics},
volume = {31},
number = {2},
pages = {171-180},
year = {1980},
issn = {0012-365X},
doi = {https://doi.org/10.1016/0012-365X(80)90033-3},
url = {https://www.sciencedirect.com/science/article/pii/0012365X80900333},
author = {Paul H. Edelman}
}

@misc{bandeira2024geometricperspectiveinjectivenorm,
      title={A Geometric Perspective on the Injective Norm of Sums of Random Tensors}, 
      author={Afonso S. Bandeira and Sivakanth Gopi and Haotian Jiang and Kevin Lucca and Thomas Rothvoss},
      year={2024},
      eprint={2411.10633},
      archivePrefix={arXiv},
      primaryClass={math.PR},
      url={https://arxiv.org/abs/2411.10633}, 
}

@Book{zbMATH06638014,
 Author = {Gurau, Razvan},
 Title = {Random tensors},
 ISBN = {978-0-19-878793-8},
 Year = {2017},
 Publisher = {Oxford: Oxford University Press},
 Language = {English},
 DOI = {10.1093/acprof:oso/9780198787938.001.0001},
 zbMATH = {6638014},
 Zbl = {1371.81007}
}

@article{Speicher1994,
author = {Speicher, Roland},
journal = {Mathematische Annalen},
number = {4},
pages = {611-628},
title = {Multiplicative functions on the lattice of non-crossing partitions and free convolution.},
url = {http://eudml.org/doc/165189},
volume = {298},
year = {1994},
}

@article{BelinschiMaiSpeicher+2017+21+53,
url = {https://doi.org/10.1515/crelle-2014-0138},
title = {Analytic subordination theory of operator-valued free additive convolution and the solution of a general random matrix problem},
author = {Serban T. Belinschi and Tobias Mai and Roland Speicher},
pages = {21--53},
volume = {2017},
number = {732},
journal = {Journal für die reine und angewandte Mathematik (Crelles Journal)},
doi = {doi:10.1515/crelle-2014-0138},
year = {2017},
lastchecked = {2024-11-30}
}

@article{VOICULESCU1986323,
title = {Addition of certain non-commuting random variables},
journal = {Journal of Functional Analysis},
volume = {66},
number = {3},
pages = {323-346},
year = {1986},
issn = {0022-1236},
doi = {https://doi.org/10.1016/0022-1236(86)90062-5},
url = {https://www.sciencedirect.com/science/article/pii/0022123686900625},
author = {Dan Voiculescu}
}

@article {capitaine2016spectrumdeformedrandommatrices,
    AUTHOR = {Capitaine, Mireille and Donati-Martin, Catherine},
     TITLE = {Spectrum of deformed random matrices and free probability},
   JOURNAL = {Advanced topics in random matrices},
    VOLUME = {Panor. Synthèses, 53},
      YEAR = {2017},
     PAGES = {151-190},
}

@article {MR1094052,
    AUTHOR = {Voiculescu, Dan},
     TITLE = {Limit laws for random matrices and free products},
   JOURNAL = {Invent. Math.},
  FJOURNAL = {Inventiones Mathematicae},
    VOLUME = {104},
      YEAR = {1991},
    NUMBER = {1},
     PAGES = {201--220},
      ISSN = {0020-9910,1432-1297},
   MRCLASS = {46L50 (60F99)},
  MRNUMBER = {1094052},
MRREVIEWER = {Sh.\ A.\ Ayupov},
       DOI = {10.1007/BF01245072},
       URL = {https://doi.org/10.1007/BF01245072},
}

@book {MR1217253,
    AUTHOR = {Voiculescu, D. and Dykema, K. J. and Nica, A.},
     TITLE = {Free random variables},
    SERIES = {CRM Monograph Series},
    VOLUME = {1},
      NOTE = {A noncommutative probability approach to free products with
              applications to random matrices, operator algebras and
              harmonic analysis on free groups},
 PUBLISHER = {American Mathematical Society, Providence, RI},
      YEAR = {1992},
     PAGES = {vi+70},
      ISBN = {0-8218-6999-X},
   MRCLASS = {46L50 (46L10 47B80 60H99)},
  MRNUMBER = {1217253},
MRREVIEWER = {Roland\ Speicher},
       DOI = {10.1090/crmm/001},
       URL = {https://doi.org/10.1090/crmm/001},
}

@book {andersongz2010,
    AUTHOR = {Anderson, Greg W. and Guionnet, Alice and Zeitouni, Ofer},
     TITLE = {An introduction to random matrices},
    SERIES = {Cambridge Studies in Advanced Mathematics},
    VOLUME = {118},
 PUBLISHER = {Cambridge University Press, Cambridge},
      YEAR = {2010},
     PAGES = {xiv+492},
      ISBN = {978-0-521-19452-5},
   MRCLASS = {60B20 (46L53 46L54)},
  MRNUMBER = {2760897},
MRREVIEWER = {Terence\ Tao},
}

@article{aumale2021,
author = {Au, Benson and Cébron, Guillaume and Dahlqvist, Antoine and Gabriel, Franck and Male, Camille},
year = {2021},
month = {01},
pages = {157-179},
title = {Freeness over the diagonal for large random matrices},
volume = {49},
journal = {Annals of Probability},
doi = {10.1214/20-AOP1447}
}

@Book{zbMATH06684673,
 Author = {Mingo, James A. and Speicher, Roland},
 Title = {Free probability and random matrices},
 FSeries = {Fields Institute Monographs},
 Series = {Fields Inst. Monogr.},
 ISSN = {1069-5273},
 Volume = {35},
 ISBN = {978-1-4939-6941-8; 978-1-4939-6942-5},
 Year = {2017},
 Publisher = {Toronto: The Fields Institute for Research in the Mathematical Sciences; New York, NY: Springer},
 Language = {English},
 DOI = {10.1007/978-1-4939-6942-5},
 zbMATH = {6684673},
 Zbl = {1387.60005}
}

@article{gurau2024quantumgravityrandomtensors,
      title={Quantum Gravity and Random Tensors}, 
      author={Razvan Gurau and Vincent Rivasseau},
      year={2024},
      journal={arXiv:2401.13510},
      archivePrefix={arXiv},
      primaryClass={hep-th},
      url={https://arxiv.org/abs/2401.13510}, 
}

@Book{nica2018lectures,
 Author = {Nica, Alexandru and Speicher, Roland},
 Title = {Lectures on the Combinatorics of Free Probability Theory},
 ISBN = {9781299405653},
 Year = {2006},
 Publisher = {Cambridge: Cambridge University Press},
 Language = {English},
 DOI = {},
 zbMATH = {},
 Zbl = {}
}

@Book{pastursccherbina2011,
 Author = {Pastur, Leonid and Shcherbina, Mariya},
 Title = {Eigenvalue Distribution of Large Random Matrices},
 ISBN = {978-0-8218-5285-9},
 Year = {2011},
 Publisher = {Mathematical Surveys and Monographs},
 Language = {English},
 DOI = {https://doi.org/10.1090/surv/171},
 zbMATH = {},
 Zbl = {}
}

@article{eugenewigner,
 ISSN = {0003486X, 19398980},
 URL = {http://www.jstor.org/stable/1970008},
 author = {Eugene P. Wigner},
 journal = {Annals of Mathematics},
 number = {2},
 pages = {325--327},
 publisher = {[Annals of Mathematics, Trustees of Princeton University on Behalf of the Annals of Mathematics, Mathematics Department, Princeton University]},
 title = {On the Distribution of the Roots of Certain Symmetric Matrices},
 urldate = {2024-12-02},
 volume = {67},
 year = {1958}
}

@article{marcenkopastur1967,
author = {Marčenko, V.A. and Pastur, Leonid},
year = {1967},
month = {01},
pages = {457-483},
title = {Distribution of eigenvalues for some sets of random matrices},
volume = {1},
journal = {Math USSR Sb}
}

@article{mingonica2004,
    author = {Mingo, James A. and Nica, Alexandru},
    title = {Annular noncrossing permutations and partitions, and second-order asymptotics for random matrices},
    journal = {International Mathematics Research Notices},
    volume = {2004},
    number = {28},
    pages = {1413-1460},
    year = {2004},
    month = {01},
    issn = {1073-7928},
    doi = {10.1155/S1073792804133023},
    url = {https://doi.org/10.1155/S1073792804133023},
    eprint = {https://academic.oup.com/imrn/article-pdf/2004/28/1413/2349421/2004-28-1413.pdf},
}

@misc{voiculescu2019hydrodynamicexercisefreeprobability,
      title={A Hydrodynamic Exercise in Free Probability: Setting up Free Euler Equations}, 
      author={Dan Voiculescu},
      year={2019},
      eprint={1902.02442},
      archivePrefix={arXiv},
      primaryClass={math.OA},
      url={https://arxiv.org/abs/1902.02442}, 
}

@article{nechitacollins2015,
author = {Collins, Benoit and Nechita, Ion},
year = {2015},
month = {09},
pages = {},
title = {Random matrix techniques in quantum information theory},
volume = {57},
journal = {Journal of Mathematical Physics},
doi = {10.1063/1.4936880}
}

@article{Bandeira_2023,
   title={Matrix concentration inequalities and free probability},
   volume={234},
   ISSN={1432-1297},
   url={http://dx.doi.org/10.1007/s00222-023-01204-6},
   DOI={10.1007/s00222-023-01204-6},
   number={1},
   journal={Inventiones mathematicae},
   publisher={Springer Science and Business Media LLC},
   author={Bandeira, Afonso S. and Boedihardjo, March T. and van Handel, Ramon},
   year={2023},
   month=jun, pages={419–487} }

@misc{avohou2024countingunotimesrotimesonotimes,
      title={Counting $U(N)^{\otimes r}\otimes O(N)^{\otimes q}$ invariants and tensor model observables}, 
      author={Remi Cocou Avohou and Joseph Ben Geloun and Reiko Toriumi},
      year={2024},
      eprint={2404.16404},
      archivePrefix={arXiv},
      primaryClass={hep-th},
      url={https://arxiv.org/abs/2404.16404}, 
}

@article{Benedetti_2021,
   title={Remarks on a melonic field theory with cubic interaction},
   volume={2021},
   ISSN={1029-8479},
   url={http://dx.doi.org/10.1007/JHEP04(2021)197},
   DOI={10.1007/jhep04(2021)197},
   number={4},
   journal={Journal of High Energy Physics},
   publisher={Springer Science and Business Media LLC},
   author={Benedetti, Dario and Delporte, Nicolas},
   year={2021},
   month=apr }

@article{Jagannath_2020,
   title={Statistical thresholds for tensor PCA},
   volume={30},
   ISSN={1050-5164},
   url={http://dx.doi.org/10.1214/19-AAP1547},
   DOI={10.1214/19-aap1547},
   number={4},
   journal={The Annals of Applied Probability},
   publisher={Institute of Mathematical Statistics},
   author={Jagannath, Aukosh and Lopatto, Patrick and Miolane, Léo},
   year={2020},
   month=aug }

@misc{arous2018landscapespikedtensormodel,
      title={The landscape of the spiked tensor model}, 
      author={Ben Arous, Gerard and Mei, Song and Montanari, Andrea and Nica, Mihai},
      year={2018},
      eprint={1711.05424},
      archivePrefix={arXiv},
      primaryClass={math.ST},
      url={https://arxiv.org/abs/1711.05424}, 
}

@article{bordenavechafai2011,
author = {Bordenave, Charles and Chafai, Djalil},
year = {2011},
month = {09},
pages = {},
title = {Around the circular law},
volume = {9},
journal = {Probability Surveys},
doi = {10.1214/11-PS183}
}

@article{collinsmingospeicher2007,
     title = {Second order freeness and fluctuations of random matrices. III: Higher order freeness and free cumulants.},
     author = {Collins, Benoit and Mingo, James A. and Sniady, Piotr and Speicher, Roland},
     journal = {Documenta Mathematica},
     volume = {12},
     year = {2007},
     pages = {1-70},
     zbl = {1123.46047},
     language = {en},
     url = {http://dml.mathdoc.fr/item/05165886}
}

@misc{collins2024freecumulantsfreenessunitarily,
      title={Free cumulants and freeness for unitarily invariant random tensors}, 
      author={Benoit Collins and Razvan Gurau and Luca Lionni},
      year={2024},
      eprint={2410.00908},
      archivePrefix={arXiv},
      primaryClass={math-ph},
      url={https://arxiv.org/abs/2410.00908}, 
}

@article{zbMATH06794703,
 author = {Arizmendi, Octavio},
 title = {{{\(k\)}}-divisible random variables in free probability},
 fjournal = {Advances in Applied Mathematics},
 journal = {Adv. Appl. Math.},
 issn = {0196-8858},
 volume = {93},
 pages = {1--68},
 year = {2018},
 language = {English},
 doi = {10.1016/j.aam.2017.09.001},
 zbMATH = {6794703},
 Zbl = {1386.46050}
}

\end{document}